\documentclass[12pt]{article}
\usepackage{graphicx} 
\usepackage[a4paper, top=2cm, bottom=2cm, right=1.5cm, left=1.5cm]{geometry} 
\usepackage{amsfonts}
\usepackage{amsthm}
\usepackage{dsfont}
\usepackage{amsmath}
\usepackage{pdfpages}
\usepackage[toc,page]{appendix}

\usepackage{mathtools}
\usepackage{stmaryrd} 
\usepackage[sortcites = true, backend=biber, style=numeric, url = false, giveninits=true]{biblatex} 
\usepackage[shortlabels]{enumitem} 
\usepackage{hyperref}
\hypersetup{
    colorlinks=true,
    linkcolor=blue,
    filecolor=magenta,
    urlcolor=cyan,
    pdftitle={Overleaf Example},
    pdfpagemode=FullScreen,
    }
\usepackage[capitalize]{cleveref} 
\usepackage{xcolor} 
\usepackage{quiver} 
\usepackage{orcidlink}
\usepackage{authblk} 
\usepackage{physics} 

\usepackage{parskip}
\usepackage{commands}

\newenvironment{keywords}
{
\begin{center}
\textbf{Keywords}\\
\vspace{0.17cm}
\begin{minipage}{14.5cm}}
{\footnotesize
\end{minipage}
\end{center}}

\newenvironment{Abstract}
{
\begin{center}
\textbf{Abstract}\\
\vspace{0.25cm}
\begin{minipage}{14.5cm}}
{\footnotesize
\end{minipage}
\end{center}}

\renewcommand{\sp}[1]
{\operatorname{span}\left\{#1\right\}}
\renewcommand{\d}{\mathrm{d}}
\newcommand{\T}{\mathrm{T}}
\newcommand{\mail}[1]{\small\href{mailto:#1}{#1}}

\DeclareMathOperator{\Ima}{Im}
\newcommand{\iprodleft}{\mathbin{\scalebox{1.7}{$\lrcorner$}}}

\newtheoremstyle{spaced} 
  {10pt} 
  {10pt} 
  {\itshape} 
  {} 
  {\bfseries} 
  {.} 
  { } 
  {} 

\theoremstyle{spaced}

\newtheorem{theorem}{Theorem}[section]
\newtheorem{proposition}[theorem]{Proposition}
\newtheorem{lemma}[theorem]{Lemma}
\newtheorem{corollary}[theorem]{Corollary}

\newtheoremstyle{spaceddef} 
  {10pt} 
  {10pt} 
  {\normalfont} 
  {} 
  {\bfseries} 
  {.} 
  { } 
  {} 

\theoremstyle{spaceddef}

\newtheorem{example}{Example}[section]
\newtheorem{remark}{Remark}[section]

\newtheorem{definition}{Definition}[section]

\begin{document}
\title{\huge The Spencer cohomology and integrability of multisymplectic structures}

\author{Manuel de León \orcidlink{0000-0002-8028-2348} \mail{mdeleon@icmat.es}}
\affil{Instituto de Ciencias Matemáticas, Campus Cantoblanco, Consejo Superior de Investigaciones Científicas, C/Nicolás Cabrera, 13–15, Madrid 28049, Spain}
\affil{Real Academia de Ciencias Exactas, Físicas y Naturales de España, C/Valverde, 22, Madrid 28004, Spain}
\affil{Department of Mathematics, School of Sciences, Great Bay University, Dongguan, Guangdong, 523000 China}

\author{\\Rubén Izquierdo-López \orcidlink{0009-0007-8747-344X} \mail{ruben.izquierdo@icmat.es}}
\affil{Instituto de Ciencias Matemáticas, Campus Cantoblanco, Consejo Superior de Investigaciones Científicas, C/Nicolás Cabrera, 13–15, Madrid 28049, Spain}

\author{\\\vspace{3pt} Manuel Lainz \orcidlink{0000-0002-2368-5853} \mail{manuel.lainz@cunef.edu}}
\affil{Department of Mathematics, CUNEF Universidad, Calle Almansa 101, 28040 Madrid,
Spain}

\date{\today}

\maketitle

\begin{Abstract}
We study the integrability problem of multisymplectic structures, by identifying them as $G$-structures. Applying the theory of Spencer cohomology, we give conditions on a multisymplectic form for it to admit a chart in which it has constant coefficients. This general study allows for a rough classification of multisymplectic structures of constant linear type, depending on the natural action of the stabilizer group. The theory is illustrated by providing a scheme for proving a Darboux theorem, which is exemplified with several relevant cases. We also build linear types of multisymplectic forms $\varpi_j$ whose flatness strictly requires a condition of order $j$. Finally, the corresponding Lie algebras are computed in the case of field theories.
\end{Abstract}


\tableofcontents

\section{Introduction}

The geometric description of classical mechanics is given by symplectic manifolds, which consists of pairs $(M, \omega)$, where $M$ is an even dimensional manifold and $\omega$ is a closed and non-degenerate $2$-form \cite{lm1987, am1978}. Similarly, great effort has been dedicated to give a geometric description of classical field theories \cite{Goldschmidt1978, garcia1976, kijowski1979}. In this context, the $2$-form is generalized to a closed and non-degenerate $(n+1)$-form (see for instance \cite{cantrijn2,binz,david1,aitor1,gimmsy,narciso}). This is the setting of multisymplectic geometry.

In general, one may consider multisymplectic manifolds independently of their connection to variational problems \cite{cantrijn1, cantrijn2}. However, from a geometric standpoint, dealing with forms of degree higher than $2$ increases the difficulty. This may be traced back to several reasons. For this text we would like to hightlight two of them:

\begin{enumerate}[\rm (i)]
    \item Up to linear isomorphism, there is a unique symplectic vector space, since every pair of non-degenerate $2$-forms are isomorphic. Unsurprisingly, this is not the case for forms of higher degree. In fact, the classification of the orbits of the $\GL(V)$ action on $\bigwedge^3 V^\ast$ is a very active area of research. See for instance \cite{Revoy88} for classifications in $\dim V = 6$ and \cite{Westwick77,Westwick79,Westwick81} for classifications in $\dim V = 7$.
    \item In a more geometric vein, the closedness of a non-degenerate $2$-form $\omega$ is the necessary and sufficient condition for the existence of Darboux coordinates, namely coordinates $(q^i, p_i)$ where the symplectic form is of constant coefficients $\omega = \d q^i \wedge \d p_i$. This fails too for multisymplectic forms: closedness is not always sufficient for the existence of flat coordinates \cite{leonid, Ryvkin-Darboux-type-theorems-2025}.
\end{enumerate}

This article advances the understanding of the second point. Since general $(k+1)$-forms on vector spaces are not classified, results guaranteeing flatness (now to be referred to as flatness theorems or Darboux theorems) are given `per linear type' (see the recent review on the subject \cite{Ryvkin-Darboux-type-theorems-2025}). As far as we are aware, the first of these results in the context of multisymplectic manifolds was by G. Martin \cite{martin} for the linear type associated to the exterior bundle of a manifold $\bigwedge^k \T^\ast M$. This result was later generalized in \cite{deleon2003tulczyjews} for the linear type used in the study of classical field theories. For more linear types we refer to \cite{maxime, laura, leonid, vanzura1, vanzura2, glrr2024revistadelarealacademiadecienciasexactasfisicasynaturales.seriea.matematicas, fg2013rev.math.phys.}.

A study of this problem in general (finding flat coordinates for a multisymplectic manifold) has not been achieved. We study this subject, by providing a general framework using the theory of $G$-structures \cite{chern}. A $G$-structure, where $G \subseteq \GL(n)$ is a subgroup, is defined as the reduction of the frame bundle of a manifold $\mathcal{F}(M)$ to a principal $G$-bundle. A key observation is that multisymplectic manifolds of `constant linear type', for a fixed linear model given by $\varpi \in \bigwedge^{k+1} V^\ast$, may be interpreted as a $G_\varpi$-structure, where $G_\varpi$ is the stabilizer group of $\varpi$.

Precisely because multisymplectic structures incorporate different geometries, the general point of view of $G$-structures results highly effective. In particular, we employ the theory of the structure tensors \cite{bernard,sternberg,fujimoto,Goldschmidt1978} of a $G$-structure, which measure the obstruction of finding flat formal coordinates (namely coordinates given by a non-necessarily convergent series). These tensors are sections of a vector bundle over $M$ that has as standard fiber $H^{2,j}(\mathfrak{g})$, the Spencer cohomology of the Lie algebra $\mathfrak{g} \subseteq V^\ast \otimes V$. It is known \cite{guillemin} that the vanishing of these sections guarantee formal flatness of any $G$-structure. Hence, a natural question is the following:
\[
\text{Is every formally flat $G$-structure flat?}
\]
As far as we know, this question remains unsolved except when $G$ acts irreducibly, is of finite type, or some other conditions of similar nature are imposed (we refer to \cite{Cattafi2020AlmostStructures} as well as Subsection \ref{subsection:flatness_results} for a discussion regarding this matter). Our work extends the understanding of this question in the case of $G = G_\varpi$, for certain $\varpi \in \bigwedge^{k+1}V^\ast$. In particular, we would like to highlight the following contributions:

\begin{enumerate}[\rm (i)]
    \item To answer the question above, the study of the natural representation of $G \subseteq \GL(V)$ on $V$ is of vital importance. The study of the $G_\varpi$-action on $V$ is given in Theorem \ref{thm:casuistic_of_representation}, where we prove that if the action of $G_\varpi$ is not irreducible, and does not decompose as a direct sum (namely the to favorable scenarios for the integrability problem), then either there is an invariant $j$-isotropic space, or there is an invariant subspace $W$ which satisfies $W^{\perp, k} = \{0\}$. This yields a rough classification of multisymplectic forms.
    \item We then study the particular case where there is a $j$-isotropic invariant subspace, and we give a general result in Theorem \ref{thm:sufficient_condition_flatness_general}. This turns out to be sufficient for studying existence of flat coordinates in a lot of scenarios. The main drawback of this result it that it requires the existence of coordinates satisfying some property. This result is improved upon for the case of $j = 1$ in Theorem \ref{thm:Darboux_1_isotropic}, where we show that the hypotheses of the previous result always hold.
    \item Another one of the main contributions is the identification of a sequence of linear types $\varpi_j \in \bigwedge^3 V^\ast$ (see Theorem \ref{thm:Theorem_counterexamples}), where $\dim V = 5 + 2j$ such that the conditions on a multisymplectic manifold $(M, \omega_j)$ to admit flat coordinates are of order $j$. In terms of the Spencer cohomology, we find $\varpi_j$ that satisfies $H^{2,j}(\mathfrak{g}_{\varpi_j}) \neq \{0\}$. Explicit examples of multisymplectic manifolds realizing these obstructions are provided in Theorem \ref{thm:explicit_counterexample}.
    \item Finally, the stabilizers are computed for the linear types involving classical field theories.
\end{enumerate}

The paper is structured as follows. In Section \ref{section:Spencer_cohomology} we introduce the theory of Cartan tableaux and Spencer cohomology, which is used throughout the rest of the text \cite{bryant}. In Section \ref{section:General_G_structures} we summarize the relevant notions and results from the theory of $G$-structures so that, in Section \ref{section:multisymplectic_G_structures} we study multisymplectic structure as $G$-structures, emphasizing in flatness results. Then, in Section \ref{section:obstruction_of_fixed_order} we build the linear types with obstructions to integrability of arbitrary order. Finally, in Section \ref{section:polarized_form} we compute the Lie algebras associated to the linear types appearing in classical field theories. We conclude in Section \ref{section:conclusions} with some conclusions and discussion for further work. The reader may also find in Appendix \ref{Appendix:Malgrange_Theorem} the proof of a result on the existence of smooth solutions to non-homogeneous systems of partial differential equations of constant coefficients, which we did not find in the literature.

\section{The Spencer cohomology of a tableau}
\label{section:Spencer_cohomology}

In this section we will recall the theory of \emph{Cartan tableaux}. The definition of a tableau
is revisited in the light of the formalism of the Spencer cohomology. For more details on tableaux, we refer the reader to \cite{bryant} (see also \cite{ivey,musso}).

Throughout this section, $V$ and $W$ denote arbitrary, real, finite dimensional vector spaces.

\begin{definition}[Tableau] A \textbf{tableau} is a subspace $\mathcal{A} \subseteq V^\ast \otimes W$.
\end{definition}

\begin{remark}[Relation to PDEs]
   Cartan tableaux encode first order linear systems of partial differential equations. Indeed, in a basis $V = \sp{v_i}$ and $W = \sp{w_j}$, a first order linear system of partial differential equations with constant coefficients on a map
   $
   f \colon V \rightarrow W
   $
   reads as
   \begin{equation}
   \label{eq:induced_equation_by_tableau}
          a^i_{jk}  \pdv{f^j}{v^i} = 0\,, \qquad \forall k\,,
   \end{equation}
   for some $a^i_{jk} \in \mathbb{R}$. By identifying $\d_v f \in V^\ast \otimes W$, the condition on $f$ solving the previous equation is equivalent to $\d_v f \in \mathcal{A}$, for all $v \in V$, where $\mathcal{A}$ is the tableau defined by
   \[
   \mathcal{A} = \left\{ c^{j}_i v^i \otimes w_j \colon a^i_{jk}c^j_{i} = 0\right\}\,.
   \]
\end{remark}

\begin{remark}[The case of Lie algebras] The case which we will mostly focus on is when $W = V$ and $\mathcal{A} = \mathfrak{g} \subseteq V^\ast \otimes V$, where $\mathfrak{g}$ is a matrix Lie algebra of a closed Lie group $G\subseteq \GL(V)$. Here, the corresponding PDE on a map $f \colon V \rightarrow V$ represents the condition on a \emph{vector field} to be an infinitesimal symmetry of the trivial $G$-structure on $V$ (for more details we refer the reader to \cite{sternberg} and Section \ref{section:General_G_structures}).
\end{remark}

\begin{definition}[Prolongations of a tableau] Let $\mathcal{A} \subseteq V^\ast \otimes W$ be a tableau. For $j \geq 0$, we define its \textbf{$j$-th prolongation} as the subspace $\mathcal{A}^{(j)} \subseteq S^{j +1} (V^\ast) \otimes W$ given by
\[
\mathcal{A}^{(j)} := \left\{\alpha \in S^{j+1}(V^\ast) \otimes W \colon \iota_{v_j} \cdots \iota_{v_1} \alpha \in \mathcal{A} \,, \forall v_1, \dots, v_j \in V \right\}\,.
\]
\end{definition}

\begin{remark}[Interpretation of the prolongation] In terms of the associated PDE, $\mathcal{A}^{(j)}$ is precisely the space of all polynomial solutions of degree $j+1$. Indeed, a polynomial map $f \colon V \rightarrow W$ solves Eq. \eqref{eq:induced_equation_by_tableau} if and only if
\[
a^i_{jk} \pdv{^kf^j}{v^{i_1} \cdots \partial v^{i_j}} = 0 \,, \qquad \forall k\,,
\]
which is precisely the defining condition on $\mathcal{A}^{(j)}$, after identifying $S^{j+1}(V^\ast) \otimes W$ as the $W$-valued polynomial maps of degree $(j+1)$. When $\mathcal{A} = \mathfrak{g}$, where $\mathfrak{g}$ is the Lie algebra of some closed Lie group $G \subseteq \GL(V)$, $\mathfrak{g}^{(j)}$ is just the space of polynomial vector fields of degree $(j+1)$ that define an infinitesimal symmetry of the trivial $G$-structure.
\end{remark}

We now define the Spencer complex, which plays an important role in determining whether a $G$-structure is formally flat \cite{fujimoto}. It also has important implications for linear systems of PDEs \cite{bryant}. Even though the case we are interested in is that of $G$-structures, the general definition of Spencer cohomology in the context of Cartan tableaux will result usefull in Section \ref{section:obstruction_of_fixed_order}.

\begin{definition}[Spencer complex]
Define the \textbf{Spencer complex} of a tableau as follows; in degree $j$, it consists of the spaces of $l$-skew symmetric maps into $W$-valued polynomials of degree $j$ that solve the associated PDE, namely
\[C^{l, j} (\mathcal{A}) := \bigwedge^l V^\ast \otimes \mathcal{A}^{(j-1)}\,.\] The Spencer coboundary
$
\partial \colon C^{ l, j} \rightarrow C^{l+1, j-1}
$is defined on decomposable elements as
\[\partial \left ( \beta \otimes(\alpha_1 \odot\cdots \odot \alpha_j)  \otimes v  \right) := \sum_{i = 1}^j (\alpha_i  \wedge \beta) \otimes(\alpha_1 \odot \cdots \odot \widehat{\alpha}_i \odot \cdots \odot \alpha_j )\otimes  v\,,\]
and extended to $C^{l,j}(\mathcal{A})$ via linearity in the usual manner (here, $\odot$ represents the symmetric product).
\end{definition}

The Spencer coboundary of an element of the complex acts on polynomial vector valued forms as follows \cite{fujimoto}:
\begin{proposition}
\label{prop:explicit_description_of_partial}
Let $\beta \in C^{l,j}(\mathcal{A})$. Then, the evaluation of $\partial \beta \in C^{l+1, j-1}$ in elements in $\bigwedge^{l+1} V \otimes S^{j}(V)$ is given by
\[
(\partial \beta)(v_0 \wedge \cdots \wedge v_l, w_1 \odot \cdots \odot w_{j-1}) = \sum_{i = 0}^{l} (-1)^i \beta(v_0 \wedge \cdots \wedge \hat{v}_i \wedge \cdots \wedge v_l, v_0 \odot w_1 \odot \cdots \odot w_{j-1})\,.
\]
\end{proposition}

Furthermore, it is straight-forward to verify

\begin{proposition}[\cite{spencer1}] We have $\partial^2 = 0$.
\end{proposition}

\begin{remark}[De Rham differential]
\label{remark:De_Rham_identification} In the presence of the trivial tableau $\mathcal{A} = V^\ast \otimes W$,
we have a natural identification of $C^{l, j}(\mathcal{A})$ with the space of $l$-differential forms on $V$ with values in $W$ and with homogeneous polynomial coefficients of degree $j$. It is easy to check that with this identification, the Spencer coboundary $\partial$ corresponds to the de Rham differential $\dd$.
\end{remark}

\begin{definition}[The Spencer cohomology groups]
The cohomology of the Spencer complex
\[
H^{l,j} (\mathcal{A}) := \frac{\ker\left(C^{l,j}(\mathcal{A})\xrightarrow{\partial} C^{l+1, j-1}(\mathcal{A})\right)}{\operatorname{Im}\left(C^{l-1, j+1}(\mathcal{A}) \xrightarrow{\partial} C^{l,j}(\mathcal{A})\right)}\,.
\]
is called the {\bf Spencer cohomology}.
\end{definition}

We now give some elementary properties of the cohomology of tableaux that will be useful in the sequel.

\begin{definition}[Morphism of Cartan tableaux]
\label{def:morphism_of_tableux} Let $\mathcal{A}_1 \subseteq V_1^\ast \otimes W_1$ and $\mathcal{A}_2 \subseteq V_2^\ast \otimes W_2$ be two tableaux. A morphism $(f, g) \colon \mathcal{A}_1 \rightarrow \mathcal{A}_2$ consists of a pair of linear maps
\[
f \colon V_2 \rightarrow V_1 \qquad \text{and} \qquad g \colon W_1 \rightarrow W_2\,,
\]
such that the induced map $(f^\ast \otimes g) \colon V_1^\ast \otimes W_1 \rightarrow V_2^\ast \otimes W_2$ satisfies $(f^\ast \otimes g)(\mathcal{A}_1) \subseteq \mathcal{A}_2$.
\end{definition}

\begin{proposition}
\label{prop:properties_morphism_tableux} Let $(f, g) \colon \mathcal{A}_1 \rightarrow \mathcal{A}_2$ be a morphism of Cartan tableaux. Then, we have the following:
\begin{enumerate}[\rm (i)]
    \item The maps
    \[
    f^\ast \otimes g \colon S^j (V_1^\ast) \otimes W_1 \rightarrow S^j (V_2^\ast) \otimes W_2
    \]
satisfy $(f^\ast \otimes g)(\mathcal{A}_1^{(j)}) \subseteq \mathcal{A}_2^{(j)}$.
\item The induced maps
\[
f^\ast \otimes (f^\ast \otimes g) \colon \bigwedge^l V_1^\ast \otimes \mathcal{A}_1^{(j)} \rightarrow  \bigwedge^l V_2^\ast \otimes \mathcal{A}_2^{(j)}
\]
are cochain morphisms.
\end{enumerate}
\end{proposition}

In particular, we get an induced morphism between the Spencer cohomologies:
\[
(f,g)_\ast \colon H^{l,j}(\mathcal{A}_1) \rightarrow H^{l,j}(\mathcal{A}_2)\,.\]
This morphism satisfies the usual functorial properties. In particular, and interestingly for us, we have the following corollary:

\begin{corollary}
\label{corollary:injective_mapping_homology}
Let $\mathcal{A} \subseteq V^\ast \otimes W$ and $\widetilde{\mathcal{A}} \subseteq \widetilde V^\ast \otimes \widetilde W$ be two Cartan tableaux. Suppose that there are linear maps
\begin{align*}
    &i_1 \colon V \rightarrow  \widetilde V\,,  &r_1 \colon \widetilde V \rightarrow V\\
    &i_2 \colon \widetilde W \rightarrow W \,,  &r_2 \colon W \rightarrow \widetilde W\,,
\end{align*}
satisfying $r_j \circ i_j = \operatorname{id}$, for $j = 1, 2$ and such that both
\[
(i_1, i_2) \colon \widetilde{\mathcal{A}} \rightarrow \mathcal{A} \qquad \text{and} \qquad (r_1, r_2) \colon \mathcal{A} \rightarrow \widetilde{\mathcal{A}}
\]
are well defined Cartan tableaux morphisms. Then, the induced map in cohomology
$
(i_1, i_2)_\ast \colon H^{l,j}(\widetilde{\mathcal{A}}) \rightarrow H^{l,j}(\mathcal{A})
$
is injective.
\end{corollary}
\begin{proof} Indeed, by functoriality, we have $(r_1, r_2)_\ast \circ (i_1, i_2)_\ast = \operatorname{id}_\ast = \operatorname{id}$. Thus, $(i_1, i_2)$ is necessarily injective.
\end{proof}

\begin{definition}[Retraction] When we are in the hypotheses of Corollary \ref{corollary:injective_mapping_homology}, we say that $\widetilde{\mathcal{A}}$ is a \textbf{retraction} of $\mathcal{A}$.
\end{definition}

In practice, we will identify retractions of a Cartan tableau as follows:

\begin{lemma}
\label{lemma:sufficient_retraction_condition}
Let $\mathcal{A} \subseteq V^\ast \otimes W$ and let $\widetilde V^\ast \subseteq V^\ast$ and $\widetilde W \subseteq W$ be subspaces. Let $\widetilde{\mathcal{A}} := \mathcal{A} \cap (\widetilde V^\ast \otimes \widetilde W)$. Then, if there are complementary subspaces
\[
V^\ast = \widetilde V^\ast \oplus \widehat V^\ast \,, \qquad W = \widetilde W \oplus \widehat W
\]
such that
\[
\mathcal{A}  = \widetilde{\mathcal{A}} \oplus \mathcal{A} \cap \left(\widehat V^\ast \otimes \widetilde W + \widehat{V}^\ast \otimes \widehat W + \widetilde V^\ast \otimes \widehat{W} \right)\,,
\]
we have that $\widetilde{\mathcal{A}}$ is a retraction of $\mathcal{A}$.
\end{lemma}
\begin{proof} By defining $i_1$ and $i_2$ as the (dual of) the inclusions and $r_1$ and $r_2$ as the projections given naturally by the direct sum decomposition, we have that $(i_1, i_2) \colon \widetilde{\mathcal{A}} \rightarrow \mathcal{A}$ is a morphism of Cartan tableaux by definition and that $(r_1, r_2) \colon \mathcal{A} \rightarrow \widetilde{\mathcal{A}}$ is a morphism of Cartan tableaux by hypothesis.
 \end{proof}

\section{An introduction to the theory of \texorpdfstring{$G$}{G}-structures}
\label{section:General_G_structures}

In this section we summarize the theory of $G$-structures, in particular, defining the structure tensors. This notion will be applied in Section \ref{section:multisymplectic_G_structures} to the particular case of a $G$-structure defined by a skew-symmetric tensor. These correspond to multisymplectic manifolds of constant linear type (see \cite{fujimoto,leonid}).

\noindent For a smooth $n$-dimensional manifold $M$ and an $n$-dimensional vector space $V$, let us denote by $\pi \colon\mathcal{F}(M) \rightarrow M$ the frame bundle over $M$ defined as
\[
\mathcal{F}(M) := \{\phi \colon V \rightarrow T_x M : \phi \text{ is a linear isomorphism}, \forall x \in M\}\,.
\]
Notice that $\GL(V)$ acts on $\mathcal{F}(M)$ on the right freely and properly, defining a $\GL(V)$-principal bundle structure on $\pi \colon \mathcal{F}(M) \rightarrow M$. If $(x^\mu)$ denote coordinates on $M$, and $e^i$ denotes a basis of $V$, we have natural coordinates $(x^i, A^\mu_j)$ representing the frame
\[
\phi(e_i) = A_\mu^j \pdv{x^\mu}\,,
\]
where the coefficients $A^\mu_j$ are subject to the restriction $\det (A_\mu^j) \neq 0$, so that $\phi$ actually defines a linear isomorphism.

\begin{remark} If $V = \mathbb{R}^n$, the frame bundle $\mathcal{F}(M)$ can be equivalently identified as the bundle of bases of $\T_x M$ by considering the image of the canonical basis.
\end{remark}

\begin{definition}[$G$-structure]
    Given a subgroup $i\colon G \hookrightarrow \GL(V)$, a \textbf{$G$-structure} is a reduction of the frame bundle to a principal $G$-bundle. Namely, a subbundle $B \subseteq \mathcal{F}(M)$ which is a principal bundle with structure group $G$, where the action of $G$ on $B$ is given by the restriction of the action of $\GL(V)$ on $\mathcal{F}(M)$.
\end{definition}

\subsection{Structure tensors}

\noindent Let $M$ be a smooth manifold equipped with a $G$-structure. As we mentioned in the introduction, one of the fundamental problems is to decide whether $M$ is locally isomorphic to the flat $G$-structure on $V$. As we can find in \cite{sternberg}, the formal obstructions to integrability of a $G$-structure are tensors that take values in vector bundles with standard fibre $H^{2, j}(\mathfrak{g})$, for $j = 0, 1, \dots$ and so on, where $H^{2,j}(\mathfrak{g})$ denotes the Spencer cohomology of $\mathfrak{g} \subseteq V^\ast \otimes V$. It is the objective of the next subsection to present these obstructions. It is clear that a $G$-structure is flat if and only if there exists an adapted connection $\nabla$ on $M$ which is flat and torsionless. The structure tensors capture precisely this idea, an represent the formal obstructions for a $G$-structure to admit a torsionless connection which is flat to a certain (but finite) order at a particular point $x$.

\begin{definition}[Solder form] Let $\pi \colon B \rightarrow M$ be a $G$-structure. The \textbf{solder form} is the $V$-valued $1$-form defined for a frame $\phi \in B$ and $v \in T_\phi B$ as
\[
\theta|_\phi(v) := \phi^{-1}(\pi_\ast v)\,.
\]
\end{definition}

Notice that the solder form of a $G$-structure on $M$ is nothing but the restriction of the canonical form on the linear frame bundle $\mathcal{F}(M)$ (see \cite{nomizu}). Locally, we have that
\[
\theta = B^i_\mu \d x^\mu \otimes e_i \,,
\]
where $B^i_\mu = B^i_\mu(x, A)$ denotes the inverse to $A^\mu_i$, namely, it is the unique set of coefficients satisfying $
B^i_\mu A^\mu_j = \delta^i_j$.

\noindent Let us also note that, by definition $\theta$ satisfies the following properties:
\begin{enumerate}[\rm (i)]
    \item It is semi-basic, namely $\theta(X) = 0$, for every vertical vector field
$X$ on $B$.
    \item For $g \in G$, we have $R_g^\ast \,\theta = g ^{-1} \circ \theta$\,, where $R_g \colon B \rightarrow B$ denotes the action on the right, and the right hand side denotes the canonical action of $G$ on $V$.
\end{enumerate}

\begin{remark}[Induced form on $M$]
In particular, by standard theory of principal bundles \cite[sec. 14.11]{kolarNaturalOperationsDifferential1993}, $\theta$ induces a vector valued $1$-form on $M$, where the vector bundle where it takes values in is the induced bundle $B[V]$ by the action of $G$ on $V$. As it turns out, we have the identification $B[V] = \T M$, and the induced form $\widetilde \theta$ is the identity, identified as a section $\widetilde \theta \in \Gamma(\T^\ast M \otimes \T M)$.
\end{remark}

Principal connections on $\mathcal{F}(M) \rightarrow M$ correspond to linear connections on $\T M$; in general, we have:

\begin{definition}[$G$-connection] Let $\pi \colon B \rightarrow M$ be a $G$-structure. A principal connection on $B$ is referred to as a \textbf{$G$-connection}.
\end{definition}

\begin{remark} In fact, from a $G$-connection on $B$, given the morphism of principal bundles
\[
B \hookrightarrow \mathcal{F}(M)\qquad \text{and} \qquad G \hookrightarrow \GL(V)\,,
\]
we have an induced connection on $\mathcal{F}(M)$ and, hence, a linear connection. Conversely, if a linear connection on $M$ arises from a $G$-connection, we say that the connection is \emph{adapted} to the $G$-structure.
\end{remark}

\begin{definition}[Torsion] Let $\alpha \in \Omega^1(B, \mathfrak{g})$ be a principal connection. The \textbf{torsion} of $\alpha$ is the $V$-valued $2$-form given by the covariant derivative of $\theta$, namely $T_\alpha := \d_{\alpha} \theta$.
\end{definition}

\begin{remark}[Interpretation on $M$] Again, by standard theory of principal bundles, $\d_\alpha \theta$ is semi-basic and equivariant, so that it induces a $(B[V] = \T M)$-valued $2$-form on $M$. If we denote by $\nabla_\alpha$ the induced linear connection on $M$ by $\alpha$, the induced form is precisely $\d_{\nabla_\alpha} \operatorname{Id}$, where $\operatorname{Id} \in \Gamma(\T^\ast M \otimes \T M)$ denotes the identity.
\end{remark}

We will now discuss the local expressions of the solder form and the torsion. In the case of the natural $\GL(V)$-structure on $M$, in coordinates we have that
\[
\d \theta = \d B^i_\mu \wedge \d x^\mu \otimes e_i = \pdv{B^i_\mu}{A^{\nu}_k} \d A^{\nu}_k \wedge \d x^\mu \otimes e_i =  - B^k_\mu B^i_\nu \d A^{\nu}_k \wedge \d x^\mu \otimes e_i\,,
\]
where the identity $ \pdv{B^i_\mu}{A^{\nu}_k} = - B^k_\mu B^i_\nu$ follows by taking the derivative of $B^i_\mu A^\mu_k = \delta^i_k$.

Let $e^m_j = e^m \otimes e_j \in \gl(V)$ denote the natural set of generators and 
\[
\alpha = B^j_\rho \left(\d A^{\rho}_m + A^\nu_m\Gamma^{\rho}_{\mu \nu} \d x^\mu \right) \otimes e^m_j
\]
denotes a principal connection $1$-form, where $\Gamma^{\rho}_{\mu \nu}$ only depend on $(x^\mu)$. Then, a straightforward computation shows 
\[
\d _\alpha \theta = B^i_\rho \Gamma^{\rho}_{\nu \mu} \d x^\nu \wedge \d x^\mu \otimes e_i\,,
\]
so that we have $\d _\alpha \theta = 0$ if and only if $\Gamma^{\rho}_{\nu \mu} = \Gamma^\rho_{\mu \nu}$, so that we recover the usual definition of torsionless connection.

\noindent In general, the torsion of a principal connection is not intrinsic, since it depends on the chosen connection. However, we can quotient out $\bigwedge^2 \T^\ast M \otimes \T M$ by a particular subbundle in such a way that the \emph{projected} torsion is intrinsic and independent of the choice of connection made. This bundle is obtained as follows. First notice that the natural action of $G$ on $V^\ast \otimes V$ induces the adjoint action on $\mathfrak{g} \subseteq V^\ast \otimes V$. This implies that all of the notions defined in the previous section remain $G$-equivariant. In particular:
\begin{enumerate}[\rm (i)]
    \item The induced action of $G$ on $\bigwedge^l V^\ast \otimes S^{j+1}(V^\ast) \otimes V$ restricts to an action on the Spencer Complex $C^{l,j}(\mathfrak{g})$.
    \item The coboundaries $\partial \colon C^{l,j}(\mathfrak{g}) \rightarrow C^{l+1,j-1}(\mathfrak{g})$ are equivariant with respect to the $G$-action.
    \item Hence, $G$ acts on the Spencer cohomology groups $H^{l,j}(\mathfrak{g})$.
\end{enumerate}

\noindent From the inclusion $\mathfrak{g} \subseteq V^\ast \otimes V$, we get an inclusion of the induced vector bundles, so that the adjoint bundle $B[\mathfrak{g}]$ of $B$ is a subbundle of $\T^\ast M \otimes \T M$
\[
B[\mathfrak{g}] \hookrightarrow \T^\ast M \otimes \T M\,.
\]
In general, by the same argument and with the obvious notations, we have
\[
B[\mathfrak{g}^{(j)}]\hookrightarrow  S^{j+1}(\T^\ast M) \otimes \T M\,,
\]
and from the equivariance of the Spencer coboundary $\partial$, we get induced maps
\[
\widetilde \partial \colon \bigwedge^{l} \T^\ast M \otimes B[\mathfrak{g}^{(j)}] \longrightarrow  \bigwedge^{l+1} \T ^\ast M \otimes B[\mathfrak{g}^{(j-1)}]\,.
\]

\begin{proposition}[\cite{sternberg}] Let $\alpha_1$ and $\alpha_2$ be principal connections on $B \rightarrow M$. Then, $\alpha = \alpha_1 - \alpha_2$ is an equivariant semi-basic $1$-form on $B$. Let $\widetilde \alpha \in \Omega^1(B[\mathfrak{g}])$ be the induced $1$-form on $M$. Then, we have
\[
\widetilde T_{\alpha_1} - \widetilde T_{\alpha_2} = \widetilde \partial \widetilde \alpha\,.
\]
\end{proposition}
\begin{proof} The proposition follows straightforwardly in coordinates. Indeed, let $\alpha_1$ and $\alpha_2$ be given by
\[
\alpha_1 = B^j_\rho \left(\d A^{\rho}_m + A^\nu_m\Gamma^{\rho}_{\mu \nu} \d x^\mu \right) \otimes e^m_j \qquad \text{and} \qquad \alpha_2 = \widetilde B^j_\rho \left(\d A^{\rho}_m + A^\nu_m\widetilde \Gamma^{\rho}_{\mu \nu} \d x^\mu \right) \otimes e^m_j\,,
\]
in such a way that $B^j_\rho A^\nu_m\Gamma^{\rho}_{\mu \nu} e^m_j$ and $B^j_\rho A^\nu_m\widetilde \Gamma^{\rho}_{\mu \nu} e^m_j$ lie in $\mathfrak{g} \subseteq V^\ast \otimes V$. Their respective torsions are
\[
T_{\alpha_1} = \d _{\alpha_1} \theta = B^i_\rho \Gamma^{\rho}_{\nu \mu} \d x^\nu \wedge \d x^\mu \otimes e_i \qquad \text{and} \qquad T_{\alpha_2} = \d _{\alpha_2} \theta = B^i_\rho \widetilde \Gamma^{\rho}_{\nu \mu} \d x^\nu \wedge \d x^\mu \otimes e_i\,.
\]
The difference of the connections is $\alpha = B^j_\rho A^\nu_m\left( \Gamma_{\mu \nu}^\rho - \widetilde \Gamma_{\mu \nu}^\rho \right) \d x\mu \otimes e^m_j$. The induced sections of $\bigwedge \T^\ast M \otimes \T M$ and $B[\mathfrak{g}]$ are expressed in coordinates as
\begin{align*}
    \widetilde T_{\alpha_1} =  \Gamma^\rho_{\nu \mu} \d x^\nu \wedge \d x^\mu \otimes \pdv{x^\rho} \,, \qquad \widetilde T_{\alpha_2} =  \widetilde \Gamma^\rho_{\nu \mu} \d x^\nu \wedge \d x^\mu \otimes \pdv{x^\rho} \qquad \\
    \text{and} \qquad \widetilde \alpha = \left( \Gamma_{\nu \mu}^\rho - \widetilde \Gamma_{\nu \mu}^\rho \right) \d x^\nu \otimes \d x^\mu \otimes \pdv{x^\rho}\,,
\end{align*}

which clearly satisfy the desired relation.
\end{proof}

\noindent In particular, if $\alpha$ is a principal connection, the image of the torsion under the projection
\[
p_0 \colon \bigwedge^2 \T ^\ast M \otimes \T M \rightarrow \bigwedge^2 \T ^\ast M \otimes \T M / \Im \widetilde \partial \cong B[H^{2, 0} (\mathfrak{g})]
\]
is independent of $\alpha$.

\begin{definition}[Structure tensor of order zero] The obtained section, $c_0 := p_0(\widetilde T_\alpha) \in \Gamma(B[H^{2,0}(\mathfrak{g})])$ is called the \textbf{structure tensor of order zero} (also referred to as the \textbf{intrinsic torsion}).
\end{definition}

\noindent As a corollary of this discussion, if a $G$-structure on $M$ admits a torsionless adapted connection, we have $c_0 = 0$ identically. In fact, the converse also holds:

\begin{theorem}[\cite{nagano}] Let $M$ be a smooth manifold. Then, a $G$-structure admits a torsionless adapted connection if and only if the structure tensor of order zero vanishes identically.
\end{theorem}

\begin{remark} A different perspective on the existence of torsionless connections will result useful to generalize the structure tensor of order zero to higher order. Notice that $\mathcal{F}(M)$ is actually an open subbundle of the jet manifold
\[
\mathcal{F}(M) \hookrightarrow J^1_0(V, M)\,,
\]
where $J^1_0(V, M)$ denotes the first jets of maps $f \colon V \rightarrow M$ with source $0 \in V$. Hence, connections on the bundle $\mathcal{F}(M) \rightarrow M$ correspond to sections of an iterated jet bundle:
\[
J^1 \mathcal{F}(M) \hookrightarrow J^1 J^1_0(V, M)\,.
\]
However, we should remark that the first jet bundle corresponds to jets of maps $V \rightarrow M$ and the second jet refers to sections of $J^1_0(V, M)$.
\end{remark}

It turns out that there is a way of identifying a particular subspace of $J^{k} \mathcal{F}(M)$ as the invertible jets in $J^{k+1}_0(V, M)$.

\begin{definition}[Frame bundle of order $k$] Let us define the \emph{frame bundle of order $k$} as \[\mathcal{F}^{(k)}(M) := \{j^{k}_0 f \colon f \colon V \rightarrow M \text{ local diffeomorphism}\}\,.\]
\end{definition}

The frame bundle of order $k$ is naturally a principal bundle over $M$ (see \cite{deleon1989framebundles, kolarNaturalOperationsDifferential1993}).

\begin{proposition}[\cite{ Libermann1997}] Let $f \colon V \rightarrow M$ be a local diffeomorphism around $0 \in V$ and let $\psi := f^{-1}$ denote its local inverse. With the section $j^1 f \circ \psi^{-1} \in \Gamma(\mathcal{F}(M))$, define the map $\beta \colon \mathcal{F}^{(q+1)}(M) \rightarrow J^{q} \mathcal{F}(M)$ as
\[
\beta (j^{q+1}_0 f) := j^{q}_{f(x)} (j^1 f \circ \psi^{-1})\,.
\]
Then, $\beta$ is an embedding.
\end{proposition}

The image of $\beta$ is called the space of \emph{holonomic jets} and denoted by $H^{q}(M)$. Then, from this perspective, the structure tensor of order zero of a $G$-structure $\pi \colon B \rightarrow M$ at a point $x \in M$ measures whether we have $J^1 B \cap H^2(M) \neq \varnothing$ at that particular point $x$, when we identify $J^1 B \subset J^1 \mathcal{F}(M)$. Equivalently, whether there is a local diffeomorphism $f \colon V \rightarrow M$ in such a way that $j^1_0 f \in B$ and $j^2_0 f \in J^1 B$ (meaning that this chart makes the structure $1$-flat at $f(0)$, see \cite{bryant}).

The perspective above greatly simplifies the definition of the jets of higher order, which measure whether the intersections $J^k B \cap H^{k}(M)$ are empty or not.

\begin{definition}[$k$-flatness] A $G$-structure $\pi \colon M \rightarrow M$ is called \emph{$k$-flat} at $x \in M$ if the following intersection is non-vanishing:
\[
J^k B |_{x} \cap H^q(M)|_x \neq \varnothing\,.
\]
\end{definition}

Notice that if a $G$-structure is $k$-flat, it is $m$-flat, for every $m \leq k$.

\begin{definition}[Holonomic jets] The space of \emph{holonomic $k$-jets} of a $G$-structure is defined as
\[
H^k B := J^m B \cap H^k(M)\,.
\]
\end{definition}

Notice that the diffeomorphism $\beta \colon \mathcal{F}^{(k+1)}(M) \rightarrow H^{k} (M)$ gives a map
\[
\beta^{-1}|_{H^{k} B} \colon H^k B \rightarrow \mathcal{F}^{(k+1)}(M)\,.
\]
We shall denote the image of this map by $\mathcal{F}^{(k+1)}_G(M)$. It corresponds to jets of diffeomorphisms making the $G$-structure flat up to order $(k+1)$ at certain point $x$.

\begin{proposition} Let $\pi \colon B \rightarrow M$ be a $k$-flat $G$-structure. Then, $\mathcal{F}^{(m+1)}_G(M)$ is a manifold, and is a principal subbundle of $\mathcal{F}^{(m)}(M)$ with structure group $G \ltimes \left(\mathfrak{g}^{(1)} \oplus \cdots \oplus \mathfrak{g}^{(m)} \right)$.
\end{proposition}

\begin{proof} It follows from the observation that given two jets $j^{m+1}_0 f$ and $j^{m+1}_0 g$ with the same target, the map $g^{-1} \circ f \colon V \rightarrow V$ is a symmetry of the trivial $G$-structure $V \times G$ of order $(m+1)$. In particular, $j^{m+1}_0(g^{-1} \circ f) \in G \ltimes \left(\mathfrak{g}^{(1)} \oplus \cdots \oplus \mathfrak{g}^{(m)} \right)$.
\end{proof}

Now, the definitions of the higher order solder forms follow from the fact that on $\mathcal{F}^{(k)}(M)$ there is a canonical $1$-form generalizing the usual solder form. Let $j^{m+1}_0 f \in \mathcal{F}_G^{(m+1)}(M)$. Since $f$ is a diffeomorphism of $G$-structures to order $(m+1)$, composition induces a diffeomorphism
\[
j^{m}f \colon \mathcal{F}^{(m)}(V) \rightarrow \mathcal{F} ^{(m)} (M)\,,
\]
whose pushforward restricts to a linear isomorphism
\[
(j^m f)_\ast \colon \T_{(0, e, 0, \dots, 0)} \mathcal{F}^{(m)}_G (V) \rightarrow \T_{j^{m}_x f} \mathcal{F}^{m}_G(M)\,.
\]
Using the identification $\T_{(0, e, 0, \dots, 0)} \mathcal{F}^{(m)}_G (V) = V \oplus \mathfrak{g} \oplus\mathfrak{g}^{(1)} \oplus \cdots \oplus \mathfrak{g}^{(m-1)}$, we define:

\begin{definition}[Canonical form on $\mathcal{F}_G^{(k+1)}(M)$] Let $M$ be endowed with a $k$-flat $G$-structure. The \emph{solder form of order $k$} is defined as follows: For $j^{k+1}_0 f \in \mathcal{F}^{k+1}_G(M)$ and $v \in \T \mathcal{F}^{k+1}_G(M)$ we define
\[
\theta^{(k+1)}|_{j^{k+1}_0 f}(v) := (j^k f)_\ast^{-1} ((\pi^{k+1}_k)_\ast v)\,.
\]
We have that $\theta \in \Omega^1\left(\mathcal{F}^{(k+1)}(M), V \oplus \mathfrak{g} \oplus\mathfrak{g}^{(1)} \oplus \cdots \oplus \mathfrak{g}^{(k-1)}\right)$.
\end{definition}


Now the definition of the higher order structure tensor follows the same scheme as the definition of the structure tensor of order zero: We have that $ \theta^{(k)}$ is an equivariant $1$-form, so that we can take the covariant derivative with respect to any connection. The main difference with the discussion above, however, is that we will restrict to connections which are \emph{holonomic}, namely connection $1$-forms $\alpha$ on $\mathcal{F}_G^{(k)} (M)$ such that $\d_\alpha \theta^{(m-1)} = 0$.

\begin{proposition}[\cite{guillemin}] Let $\pi \colon B \rightarrow M$ be a $k$-flat $G$-structure and let $\alpha_1$ and $\alpha_2$ be holonomic principal connections on $\pi^{k}_{-1} \colon\mathcal{F}^{(k+1)}_B\rightarrow M$, identified as $(\mathfrak{g}\oplus \cdots \oplus \mathfrak{g}^{(k)})$-valued $1$-forms. Their difference $\alpha_1 - \alpha_2$ can be identified (such as before) as an element
\[
\alpha = \alpha_1 - \alpha_2 \in \Gamma(\T^\ast M \otimes \mathfrak{g}^{(m)})\,.
\]
Then, we have
\[
\d_{\alpha_1} \widetilde \theta^{(m)} - \d_{\alpha_2} \widetilde \theta^{(m)} = \partial \alpha\,,
\]
where $\partial$ is the Spencer coboundary.
\end{proposition}


It can be shown that $\d_{\alpha} \widetilde \theta^{(m)}$ is actually a Spencer cocycle, for every holonomic $\alpha$. This gives that the following definition is independent of the connection:

\begin{definition}[Higher order structure tensor] Let $\pi \colon B \rightarrow M$ be a $G$-structure which is $k$-flat and let $\alpha$ be a holonomic principal connection on the bundle $H^m B \rightarrow M$. Then, the induced section
\[
c_k = [\d_\alpha \widetilde\theta^{(m)}]^{\widetilde{\,}} \in \Gamma\left(B[H^{2,k}(\mathfrak{g})]\right)
\]
is called the \emph{structure tensor of order $k$}.
\end{definition}

Now, if a $G$-structure is $(k+1)$-flat, its structure tensor of order $k$ vanishes. Conversely:

\begin{theorem}[Guillemin] Let $\pi \colon B\rightarrow M$ be a $G$-structure which is $k$-flat. Then, it is $(k+1)$-flat at a point $x$ if and only if its structure tensor of order $k$ satisfies $c_k(x) = 0$.
\end{theorem}

\subsection{Flatness and formal flatness}

A theorem due to Spencer \cite{spencer1,spencer2} proves that there are only finitely many $H^{2,k}(\mathfrak{g})$ which are nonzero. Hence, whether a $G$-structure is $k$-flat, for every $k$ can be decided in a finite amount of steps.

\begin{definition}[Formally flat] A $G$-structure is called \emph{formally flat} if it is $k$-flat, for every $k$.
\end{definition}

Then, the discussion above may be summarized as:

\begin{theorem}[\cite{guillemin}] Let $M$ be a smooth manifold. Then, a $G$-structure is \emph{formally flat} if and only if all structure tensors vanish.
\end{theorem}

\noindent Formal flatness may be equivalently defined as the existence of coordinates given by a formal power series which makes the structure flat. In particular, a necessary condition for a $G$-structure to be locally flat is for $c_0, c_1, \dots$ to be identically zero. The converse does not necessarily hold, but the issue simplifies for $\mathfrak{g}$ of \emph{finite type}.

\begin{definition}[Finite type] A Lie algebra $\mathfrak{g} \subseteq V^\ast \otimes V$ is called of \textbf{finite type} if there is a $j \geq 1$ such that $\mathfrak{g}^{(j)} = \{0\}$.
\end{definition}

Indeed, if $\mathfrak{g}$ is such that $\mathfrak{g}^{(j)} = 0$, for some $j$, in the procedure above we would have that the bundle
\[
\pi^j_{j-1} \colon \mathcal{F}^{(j)}_G(M) \rightarrow \mathcal{F}^{(j-1)}_G(M)
\]
is a diffeomorphism. In particular, there exists only one principal connection, which is adapted and with vanishing higher torsion. The last structure tensor is easily shown to be the curvature of this principal connection, which may be used to obtain one on $\pi \colon B \rightarrow M$ which is torsionless and flat. Hence,

\begin{theorem}
\label{thm:finite_type_integrability}
If $G$ is of finite type, any $G$-structure which is formally flat is flat.
\end{theorem}

For $G$ which is not of finite type, we cannot guarantee the implication above. However, if $G$ acts irreducibly, it follows (essentially by characterizing all irreducible representations of Lie algebras which are of finite type):

\begin{theorem}[\cite{singerst}]
\label{thm:irreducible_integrability}
If $G$ acts on $V$ irreducibly, any $G$-structure which is formally flat is flat.
\end{theorem}

\section{\texorpdfstring{$G$}{G}-structures defined by forms}
\label{section:multisymplectic_G_structures}

{
Let us fix a skew--symmetric $(k+1)$-tensor $\varpi \in \bigwedge^{k+1} V^\ast$ on a vector space $V$. The objective of this section is to study $G$-structures that are induced by $\varpi$.
\begin{definition}
We say that a manifold $M$ together with a differential $(k+1)$-form $\omega\in \Omega^{k+1}(M)$ \emph{has constant linear type $\varpi$} if at each $x \in M$ there exists a linear isomorphism $\phi: V \to \T_x M$ such that $\phi^* \omega|_x = \varpi$. See \cite{leonid}.
\end{definition}
\begin{remark} If a manifold $M$ together with a differential $(k+1)$-form $\omega \in \Omega^{k+1} (M)$ has constant linear type $\varpi \in \bigwedge^{k+1}V^\ast$, then the subbundle of frames
\[
\{\phi:V \rightarrow T_x M \text{ linear isomorphism} : \phi^\ast \omega|_x = \varpi\}
\]
defines a reduction of the $\GL(V)$-structure of frames to a $G_{\varpi}$-structure, where $G_{\varpi}$ is defined as
\[G_{\varpi} := \{g \in \GL(V) : g^\ast \varpi = \varpi\}\,.\]
Conversely, defining a $G_{\varpi}$-structure on $M$, say $\pi \colon B \rightarrow M$, defines a differential form $\omega \in \Omega^{k+1}(M)$ with constant linear type $\varpi$ as follows: For each $x \in M$, choose $\phi \in \pi^{-1}(x)$ and define $\omega|_x := (\phi^{-1})^\ast \varpi$. The fact that $\pi \colon B \rightarrow M$ is a $G_{\varpi}$-structure ensures that $\omega|_x$ is well defined and independent of the choice of $\phi \in \pi^{-1}(x)$. We will refer to $\pi \colon B \rightarrow M$ or $(M, \omega)$ as a $G_\varpi$-structure or as a multisymplectic manifold of constant linear type $\varpi$ throughout the text.
\end{remark}}
The Lie algebra $\mathfrak{g}_\varpi$ of $G_\varpi$ can be computed as follows:

\begin{proposition}
\label{prop:description_of_Lie_algebra}
We have $\mathfrak{g}_\varpi = \{T \in \operatorname{End}(V) = V^\ast \otimes V \colon \iota_T \varpi = 0\}$, where the contraction by $(1,1)$-tensors is defined as
\[
(\iota_T \varpi)(v_1, \dots, v_{k+1}) = \sum_{j = 1}^{k+1} \varpi\left(v_1, \dots, v_{j-1}, T(v_j), v_{j+1}, \dots, v_{k+1}\right)\,.
\]
\end{proposition}

\begin{proof} By definition, $T \in \mathfrak{g}_\varpi$ if and only if $(e^{t T})^\ast \varpi = \varpi$, for every $t \in \mathbb{R}$. This is equivalent to
\[
\dv{t}\bigg |_{t = 0} \left( (e^{t T})^\ast \varpi\right) =  0\,.
\]
Notice that
\begin{align*}
    \dv{t}\bigg|_{t = 0} \left((e^{t T})^\ast \varpi \right) (v_1, \dots, v_{k+1}) &= \dv{t} \bigg |_{t   = 0} \left( \varpi ((e^{t T}) \cdot v_1, \dots, (e^{t T}) \cdot v_{k+1})\right)\\
    &= \sum_{j = 1}^{k+1} \varpi(v_1, \dots, v_{j-1}, T (v_j), v_{j+1}, \dots, v_{k+1})\\
    &= (\iota_T \varpi)(v_1, \dots, v_{k+1})\,,
\end{align*}
from which the result follows.
\end{proof}

\begin{remark}
It will be useful to know that for decomposable tensors $T = \eta \otimes v$, we have that $\iota_{\eta \otimes v} \varpi = \eta \wedge \iota_v \varpi$.
\end{remark}

\subsection{The Spencer cohomology}

The aim of this subsection is to give an interpretation of $H^{l,j}(\mathfrak{g}_\varpi)$ for the Lie algebra $\mathfrak{g}_\varpi \subseteq V^\ast \otimes V$ associated to $G_\varpi$, where $\varpi \in \bigwedge^{k+1} V^\ast$ is a fixed arbitrary form. In particular, we find an isomorphism with the cohomology of a different complex (see Proposition \ref{prop:cohomology_isomorphism}). First, define \[
\flat_\varpi \colon V \rightarrow \bigwedge^k V^\ast \,, \qquad v \mapsto \iota_v \varpi
\]
and, by making abuse of notation, let us denote similarly
\[
\flat_\varpi \colon C^{l,j} \rightarrow \bigwedge^{l} V^\ast \otimes S^{j}(V^\ast) \otimes \bigwedge^{k} V^\ast
\]
its natural extension, namely $\flat_\varpi(\beta \otimes \alpha \otimes v) := \beta \otimes \alpha \otimes \iota_v \varpi$. Notice that now we can define two Spencer coboundaries, which we denote by $\partial$ and $\dd$, which act on the left and right spaces, respectively, and are defined for decomposable elements $\beta = \beta_1 \wedge \cdots \wedge \beta_l \in \bigwedge^l V^\ast$, $\alpha = \alpha_1 \odot \cdots \odot \alpha_j \in S^{j}(V^\ast)$ and $\gamma = \gamma_1 \wedge \cdots \wedge \gamma_k \in \bigwedge^{k} V^\ast$ as
\begin{align*}
    \partial(\beta \otimes \alpha\otimes \gamma ) &:= \sum_{i = 0}^{l} (\alpha_i \wedge \beta) \otimes (\alpha_1 \odot \cdots \odot \widehat \alpha_i \odot \cdots \odot\alpha_l) \otimes \gamma\,, \\
    \dd (\beta \otimes \alpha \otimes \gamma) &:= \sum_{i = 0}^{l} \beta \otimes (\alpha_1 \odot \cdots \odot \widehat \alpha_i \odot \cdots \odot \alpha_l) \otimes (\alpha_i \wedge \gamma)\,.
\end{align*}
The following relation follows from a straight-forward computation:
\begin{proposition}
\label{prop:commutation_partial_dd}
We have $\partial \circ \dd = \dd \circ \partial$.
\end{proposition}
\noindent Let us define
\[
\mathcal{I}^{l,j} := \left\{\widetilde \alpha \in \flat_\varpi\left(\bigwedge^l V^\ast \otimes S^{j}(V^\ast) \otimes V \right)  \colon \dd \widetilde \alpha = 0\right\}\,.
\]
\begin{remark} In the context of field theories (see \cite{cantrijn2}) it is interesting to study the space of Hamiltonian forms, which is defined as follows. For a closed $(k+1)$-form $\omega$ on a manifold $M$, a form $\alpha \in \Omega^{k-1}(M)$ is called \emph{Hamiltonian} if $\d \alpha = \iota_{X_\alpha} \omega$, for some vector field $X_\alpha$. These forms represent the classical Noether correspondence between observables and symmetries. The above space $\mathcal{I}^{l,j}$ can then be identified with the space of $l$-forms with values in `the differentials of Hamiltonian forms'. The rest of this section proves that the cohomology of the Spencer complex of $\mathfrak{g}_\varpi$ can be identified with the cohomology of "forms with values in Hamiltonian forms".
\end{remark}
In light of Proposition \ref{prop:commutation_partial_dd}, we get a complex:
\[\begin{tikzcd}[cramped]
	\cdots & {\mathcal{I}^{l,j}} & {\mathcal{I}^{l+1,j-1}} & \cdots
	\arrow["\partial", from=1-1, to=1-2]
	\arrow["\partial", from=1-2, to=1-3]
	\arrow["\partial", from=1-3, to=1-4]
\end{tikzcd}\,.\]

\begin{proposition}
\label{prop:cohomology_isomorphism} The map $\flat_\varpi$ defines a homomorphism between complexes
\[\begin{tikzcd}[cramped]
	\cdots & {\mathcal{I}^{l,j}} & {\mathcal{I}^{l+1,j-1}} & \cdots \\
	\cdots & {C^{l,j}(\mathfrak{g}_\varpi)} & {C^{l+1,j-1}(\mathfrak{g}_\varpi)} & \cdots
	\arrow["\partial", from=1-1, to=1-2]
	\arrow["\partial", from=1-2, to=1-3]
	\arrow["\partial", from=1-3, to=1-4]
	\arrow["\partial", from=2-1, to=2-2]
	\arrow["{\flat_\varpi}"{description}, from=2-2, to=1-2]
	\arrow["\partial", from=2-2, to=2-3]
	\arrow["{\flat_\varpi}"{description}, from=2-3, to=1-3]
	\arrow["\partial", from=2-3, to=2-4]
\end{tikzcd}\,.\]
Furthermore, the induced morphism in cohomology is an isomorphism.
\end{proposition}

\noindent The proof of Proposition \ref{prop:cohomology_isomorphism} follows from the following lemma

\begin{lemma}
\label{lemma:characterization_of_Imflat}
For every $\widetilde \beta \in C^{l,j}(\mathfrak{g}_\varpi)$ we have $\partial \flat_\varpi(\widetilde \beta) = \flat_{\varpi} (\partial \widetilde \beta)$. Furthermore, if $\widetilde \alpha \in \Ima \flat_\varpi \subseteq \bigwedge^l V^\ast\otimes S^{j}(V^\ast) \otimes \bigwedge^k V^\ast$, we have \[
\widetilde \alpha \in \flat_\varpi \left( C^{l,j}(\mathfrak{g}_\varpi)\right)
\]
if and only if $\dd \widetilde \alpha = 0$.
\end{lemma}

\begin{proof} The fact that $\flat_\varpi$ satisfies $\partial \circ \flat_\varpi = \flat_\varpi \circ \partial$ follows easily from the definition. In order to prove the second claim, it is enough to show that
\[
\flat_\varpi \left( \mathfrak{g}_{\varpi}^{(j-1)}\right) = \mathcal{I}^{0, j}\,,\] since
\[
\flat_\varpi \left( \bigwedge^l V^\ast \otimes \mathfrak{g}^{(j-1)}_\varpi\right) = \bigwedge^l V^\ast \otimes \flat_\varpi \left( \mathfrak{g}^{(j-1)}_\varpi\right)\,.
\]
Indeed, let $\widetilde \alpha \in \Ima \flat_\varpi\left( S^{j}(V^\ast) \otimes V \right)$ and let $\widetilde \beta \in S^{j}(V^\ast) \otimes V$ be such that $\flat_\varpi(\widetilde \beta) = \widetilde \alpha$. Then, $\widetilde \beta \in \mathfrak{g}^{(j-1)}$ if and only if for every $v_1, \dots, v_{j-1} \in V$ we have $\iota_{v_{j-1}} \cdots \iota_{v_1} \beta \in \mathfrak{g}_\varpi$ which, by Proposition \ref{prop:description_of_Lie_algebra}, is equivalent to
\[
\left( \iota_{v_{j-1}} \cdots \iota_{v_1} \beta\right) \iprodleft \varpi = 0\,.
\]
Namely, that for every $u_1, \dots, u_{k+1} \in V$
\begin{equation}
\label{eq:condition_on_beta}
    \varpi (\beta(v_1, \dots, v_{j-1}, u_1), u_2, \dots, u_{k+1}) + \cdots + \varpi(u_1, \dots, u_{k}, \beta(v_1, \dots, v_{j-1}, u_{k+1})) = 0\,.
\end{equation}
Now, we may rewrite Eq. \eqref{eq:condition_on_beta} as
\[
\sum_{i = 1} ^{k+1} (-1)^{i-1} \flat_\varpi(\beta)(v_1, \dots, v_{j-1}, u_i) (u_1, \dots, \widehat u_i, \dots, u_{k+1}) = 0\,,
\]
which using the formula given in Proposition \ref{prop:explicit_description_of_partial}, is equivalent to $\dd \flat_\varpi(\beta) = 0$.
\end{proof}

\begin{proof}[Proof of Proposition \ref{prop:cohomology_isomorphism}] We will first prove it when $\varpi$ is non-degenerate. By Lemma \ref{lemma:characterization_of_Imflat}, $\flat_\varpi$ takes values in $\mathcal{I}^{l,j}$ and it defines a homomorphism of complexes. Let us show that it is actually an isomorphism. Indeed, it is already surjective at the cochain level, so we only need to check injectivity. Let $\widetilde \beta \in C^{l,j} (\mathfrak{g}_\varpi)$ be such that $\flat_\varpi(\widetilde \beta)$ is a coboundary. Then, using again surjectivity of $\flat_\varpi$ at the cochain level, there exists $\widetilde \gamma \in C^{l-1, j+1}(\mathfrak{g}_\varpi)$ satisfying
\[
\flat_\varpi(\widetilde \beta) = \partial \flat_\varpi(\widetilde \gamma) = \flat_\varpi(\partial \widetilde \gamma) \,.
\]
In particular, $\flat_\varpi(\widetilde\beta - \partial \widetilde \gamma) = 0$ and it follows that
\[
\widetilde \beta -\partial \widetilde \gamma \in \bigwedge^l V^\ast \otimes S^{j}(V^\ast) \otimes \{v \in V \colon \iota_v \varpi = 0\}\,.
\]
Since $\varpi$ is non-degenerate, we obtain $\widetilde \beta = \partial \widetilde \gamma$.

For the degenerate case, we simply study the form induced in $V/\{v \in V : \iota_v \varpi = 0\}$, which we denote by $\varpi_0$. We have that $W = \{v \in V : \iota_v \varpi = 0\}$ is an invariant subspace of $V$ by the induced $\mathfrak{g}_\varpi$-action, so that we get an induced representation of $\mathfrak{g}_\varpi$ on $V/W$ which is, in fact, the representation of $\mathfrak{g}_{\varpi_0}$. This implies that these Lie algebras fit into a short exact sequence
\[\begin{tikzcd}[cramped]
	{\{0\}} & {V^\ast \otimes W} & {\mathfrak{g}_\varpi} & {\mathfrak{g}_{\varpi_0}} & {\{0\}} \\
	&& {V^\ast \otimes V} & {(V/W)^\ast \otimes (V/W)}
	\arrow[from=1-1, to=1-2]
	\arrow[from=1-2, to=1-3]
	\arrow[from=1-3, to=1-4]
	\arrow[hook', from=1-3, to=2-3]
	\arrow[from=1-4, to=1-5]
	\arrow[hook', from=1-4, to=2-4]
\end{tikzcd}\,.\]
Let $L$ be an arbitrary complement to $W$, $V = L \oplus W$, then we can identify $V/W \cong L$. Let us denote by
\[
i \colon V/W \rightarrow V \,, \qquad p \colon V \rightarrow V/W\,,
\]
the inclusion and projection given by the above decomposition. Then, an elementary computation shows that
\[
(p,i) \colon \mathfrak{g}_{\varpi_0} \rightarrow \mathfrak{g}_{\varpi} \,, \qquad (i,p) \colon \mathfrak{g}_{\varpi} \rightarrow \mathfrak{g}_{\varpi_0}
\]
are morphisms of tableaux. We will show that the induced morphisms in cohomology are respective inverses, which will finish the proof. It is clear that $(i,p)_\ast \circ (p,i)_\ast = \operatorname{id}$, so it only remains to show $h = (p,i)_\ast \circ (i,p)_\ast = \operatorname{id}$. Indeed, we have that for every $\alpha \in \bigwedge^l(V^\ast) \otimes \mathfrak{g}^{(j)}$, $\alpha - h(\alpha)$ takes values in $\bigwedge^l(V^\ast) \otimes S^{j+1}(V^\ast) \otimes W$. If $\alpha$ is closed, so is $\alpha - h(\alpha)$. However, the previous complex is exact so $\alpha - h(\alpha) = \partial \beta$, for some $\beta \in \bigwedge^{l-1}(V^\ast) \otimes S^{j+2}(V^\ast) \otimes W$. Since it takes values in $W$, it is trivial to show that $\beta \in \bigwedge^{l-1}(V^\ast) \otimes \mathfrak{g}^{(j+1)}_\varpi$, which proves that $h$ induces an isomorphism in the cohomology, finishing the proof.
\end{proof}

\subsection{Structure of multisymplectic representations}

Taking into account the results on the integrability of $G$-structures (Theorems \ref{thm:finite_type_integrability} and \ref{thm:irreducible_integrability}), it is natural to study representations of Lie groups which leave a form invariant.

\begin{definition}[Multisymplectic representation] A representation of a Lie group $\rho \colon G \rightarrow \GL(V)$ is called a \textbf{multisymplectic representation} if $\rho(G) \subseteq G_\varpi$, for some nondegenerate $\varpi \in \bigwedge^{k+1}V^\ast$. Equivalently, if there is a nonzero and nondegenerate invariant element of the induced representation on $\bigwedge^{k+1} V^\ast$.
\end{definition}

In this section we prove a general result of multisymplectic actions which will be helpful to adopt a general strategy to prove flatness theorems.

\begin{remark} Notice that if a representation $\rho_V: G \rightarrow \GL(V)$ of a Lie group on $V$ has an invariant subspace $W \subseteq V$, it induces two representations $\rho_W \colon G \rightarrow \GL(W)$ and $\rho_{V/W} \colon G \rightarrow \GL(V/W)$ defined in the natural way.
\end{remark}

\begin{definition} Let $\varpi \in \bigwedge^{k+1}V^\ast$ be a form and $W \subseteq V$ be a subspace. We define $W^{\perp, j}:= \{v \in V \colon \iota_{v \wedge w_1 \wedge \cdots \wedge w_{j}} \varpi = 0 \,, \forall w_i \in W\}$. A subspace is called \textbf{$j$-isotropic} if $W \subseteq W^{\perp, j}$.
\end{definition}

\begin{theorem}
\label{thm:casuistic_of_representation}
Let $\rho_V \colon G \rightarrow \GL(V)$ be a multisymplectic representation, with invariant form $\varpi \in \bigwedge^{k+1}V^\ast$. Then, if the action is not irreducible and does not decompose as the direct sum of representations, at least one of the following hold:
\begin{enumerate}[\rm (i)]
    \item There is a $j$-isotropic invariant subspace. Furthermore, if $j = 1$, we have $\mathfrak{g}_\varpi^{(1)} \cap \left( S^2(V^\ast) \otimes W\right) = \{0\}$.
    \item The action has a proper invariant subspace $\{0\} \subsetneq W \subsetneq V$ such that $W^{\perp, k} = 0$. Furthermore, $\mathfrak{g}_\varpi^{(1)} \cap \left( S^2(\operatorname{Ann}(W)) \otimes V\right) = \{0\}$, where $\operatorname{Ann}(W) \subseteq V^\ast$ denotes the annihilator of $W$.

\end{enumerate}
\end{theorem}

This motivates the following definition:
\begin{definition}[Type of a form] A form $\varpi \in \bigwedge^{k+1} V^\ast$ is said to be of
\begin{enumerate}[\rm (i)]
    \item \textbf{irreducible type}, if $G_\varpi$ acts irreducibly on $V$.
    \item \textbf{product type}, if the action of $G_\varpi$ on $V$ decomposes as a direct sum of representations $V = W_1 \oplus W_2$.
    \item \textbf{$j$-isotropic type}, if there is an invariant $j$-isotropic subspace.
    \item \textbf{{semi-finite}} type, if there is an invariant subspace $W$ such that $W^{\perp, k} = \{0\}$.
\end{enumerate}
These types are not exclusive.
\end{definition}

Then, Theorem \ref{thm:casuistic_of_representation} may be reinterpreted as stating that a form $\varpi \in \bigwedge^{k+1} V^\ast$ is of irreducible, product, $j$-isotropic or {semi-finite} type. This gives a rough characterization of the possible geometries defined by a $G_\varpi$-structure. Indeed, let us discuss case by case.

\begin{enumerate}[\rm (i)]
    \item If $\varpi$ is of irreducible type, nothing much can be said regarding the inner structure of a $G_\varpi$-structure. However, this is precisely the hypotheses of Theorem \ref{thm:irreducible_integrability}, so that any $G_\varpi$-structure which is formally flat is flat. This is the case of symplectic and volume forms, as we recover $\operatorname{Sp}(n)$ and $\operatorname{SL}(n)$-structures, respectively.
    \item If $\varpi$ is of product type, any $G_\varpi$-structure has an associated almost product structure. Then, integrability of $G_\varpi$ implies that this almost product structure is, in fact, a product structure.
    \item If $\varpi$ is of $j$-isotropic type, any $G_\varpi$-structure has a distinguished $j$-isotropic distribution. Of course, if the $G_\varpi$-structure is integrable, then the distribution is involutive. This is the case where field theories are included. Indeed, multisymplectic forms associated to a classical field theory are of $1$-isotropic type. As we shall show, in some important cases (namely when the distribution can be chosen to be $1$-isotropic), the induced representation $\rho_W$ is of finite type. Furthermore, the integrability of the leaf geometry, the integrability of the geometry transverse to the leaves, and a certain compatibility condition, which we study in the following section.
    \item If $\varpi$ is of semi-finite type, we can show that the map $\mathfrak{g} \rightarrow \rho_{W}(\mathfrak{g})$ given by restriction is injective (see Proposition \ref{prop:injective_mapping}). This means that given a frame in a leaf, there exists a unique frame extending it adapted to the $G_\varpi$-structure.
\end{enumerate}
We will continue this discussion later on.

Before giving the proof to the result above, we would like to give soem results regarding the extensions of $\mathfrak{g}_\varpi$:

\begin{lemma}
\label{lemma:semi_finite_condition}
Let $\rho \colon G \rightarrow \GL(V)$ be a multisymplectic representation, with multisymplectic form $\varpi$. Let $\{0\} \subsetneq W \subsetneq V$ be an invariant subspace. Then, we have the following:
\begin{enumerate}[\rm (i)]
    \item If $W$ is $1$-isotropic, we have $\mathfrak{g}^{(1)} \cap \left(S^2(V^\ast) \otimes W \right) = \{0\}$.
    \item If $W$ is such that $W^{\perp, k} = \{0\}$, we have that $\mathfrak{g}^{(1)} \cap \left( S^2(\operatorname{Ann}(W)) \otimes V\right) = \{0\}$.
\end{enumerate}
\end{lemma}

\begin{proof}
Without loss of generality, we may assume $G \subseteq G_\varpi$. Let us prove both scenarios:\begin{enumerate}[\rm (i)]
    \item Let $\{w_1, \dots, w_m, u_{1}, \dots, u_n\}$ denote a basis of $V$, where $W$ is generated by $w_1, \dots, w_m$. Then, since $W$ is $1$-isotropic, we have that we may write
    \[
    \varpi = \alpha_{i_1, \dots, i_{k+1}} u^{i_1} \wedge \cdots \wedge u^{i_{k+1}} + \alpha_{i_1, \dots, i_{k}} \wedge u^{i_1} \wedge \cdots \wedge u^{i_{k}}\,.
    \]
    where $u^i$ denote elements of the dual basis and
    \[
    \alpha_{i_1, \dots, i_{k+1}} \in \mathbb{R}\,, \qquad \alpha_{i_1, \dots, i_{k}} \in W^\ast\,.
    \]
    Let us show that the induced representation $\rho_W$ is of finite type. Indeed, let $f \in S^{2} V^\ast \otimes V$ denote an element $f \in \mathfrak{g}^{(1)} \subseteq \mathfrak{g}_\varpi^{(1)}$. Let us write it, in the basis above as
    \[
    f = f^{i} (u, w)\otimes  w_i + g^j(u) \otimes u_j\,,
    \]
    where $f^i$ and $g^j$ denote polynomials of degree $2$. Indeed, $g^j = g^j(u)$ since
    $W$ is invariant. By contracting $w^j$, if $f \in \mathfrak{g}_\varpi$ we have that
    \[
    \iota_{w^j} f^i \otimes w_i \in \mathfrak{g}_\varpi \,,
    \]
    which is equivalent to the following
    \[
    0 = (\iota_{w^j} \iota_{w^k} f^i) \alpha_{i_1, \dots, i_{k}} (w_i) w ^j \wedge u^{i_1} \wedge \cdots \wedge u^{i_{k}}\,.
    \]
    Hence, $(\iota_{w^j} \iota_{w^k} f^i) \alpha_{i_1, \dots, i_{k}} (w_i) = 0$, for every $j, k$ and $i_{1}, \dots, i_k$. Since $\varpi$ is nondegenerate, we obtain $f^i = 0$.
    \item
    Choose a basis $(w_j, u_i)$ of $V$, where $W = \langle w_j \rangle$. Then, \[
\varpi = \alpha + u^i \wedge \alpha_i + \cdots + u^{i_1} \wedge \cdots \wedge u^{i_m} \wedge \alpha_{i_1 \dots i_m} + \cdots + \alpha_{i_1, \dots, i_k} \wedge u^{i_1} \wedge \cdots \wedge u^{i_k}\,.
\]
Here, $(w^i, u^j)$ denotes the dual basis on $V$, and $\alpha$, $\alpha_i$, $\alpha_{i_1i_2}$, and so on, denote a $(k+1)$-form, $k$-form, $(k-1)$-form and so on on $W$. We will take $T \in \mathfrak{g} \subseteq \mathfrak{g}_\varpi \subseteq V^\ast \otimes V$ and impose $\iota_T \varpi = 0$. By requiring that every such element is such that $T(W) \subseteq W$ and that $W^{\perp, k} = \{0\}$, we will conclude the result. Indeed, the most general $T \in \mathfrak{g}_\varpi$ reads as
\[
T =B^{j_2}_{j_1} w^{j_1} \otimes w_{j_2} + u^{i_1} \otimes \left( C^{i_2}_{i_1} u_{i_2} + D^{j_1}_{i_1} w_{j_1}\right)\,,
\]
and its contraction with $\varpi$ is
\begin{align*}
    \iota_T \varpi =& +B^{j_2}_{j_1} w^{j} \wedge \iota_{w_{j_1}} \alpha\\
    &+ C^{i_2}_{i_1} u^{i_1} \wedge \alpha_{i_2} + D^{j_1}_{i_1} u^{i_1} \wedge \iota_{w_{j_1}} \alpha - B^{j_1}_{j_2} u^i \wedge \left( w^{j_2} \wedge \iota_{w_{j_1}} \alpha_i\right)\\
    & + \text{terms involving }u^{i_1} \wedge u^{i_2} \dots
\end{align*}
By requiring $\iota_T \varpi = 0$, we obtain the following condition
\begin{equation}
\label{eq:condition_symmetry}
 C^{i_2}_{i} \alpha_{i_2} + D^{j_1}_{i} \iota_{w_{j_1}} \alpha - B^{j_1}_{j_2} w^{j_2} \wedge \iota_{w_{j_1}} \alpha_i = 0\,.
\end{equation}
Now, observe that the condition $W^{\perp, k} = \{0\}$ reads as follows: If $v = A^i u_i + B^j w_j$ is a vector such that $(\iota_v\varpi)|_{W} = 0$, namely
\begin{equation}
\label{eq:condition_k_iso}
    B^j \iota_{w_j} \alpha + A^i \alpha_i = 0\,,
\end{equation}
{then we have} $A^i = B^j = 0$.

Now, let $f \in S^{2}(V^\ast) \otimes V$ be such that $f \in \mathfrak{g} \subseteq \mathfrak{g}_\varpi$. Again, since $f$ preserves $W$, we can write
    \[
    f = f^{i} (u, w) \otimes w_i + g^j(u) \otimes u_j\,.
    \]
    Contracting with $w_j$ and using that $f \in \mathfrak{g}_\varpi^{(2)}$, we obtain that $(\iota_{w_j} f^i) \otimes w_i \in \mathfrak{g}_\varpi$ which, by definition means $\iota_{(\iota_{w_j} f^i) \otimes w_i}\varpi = 0$. Computing this contraction and looking at the terms involving $u^i$ we obtain the condition
    \[
    (\iota_{u_k} \iota_{w_j} f^i) \iota_{w_i} \alpha = 0\,,
    \]
    so that $\iota_{u_k} \iota_{w_j}f^{i} = 0$ and we actually have
    $
    f^i(u, w) = \hat{f}^i(u) + \widetilde{f}^i(w)
    $. Hence, the original symmetry is
    \[
    f =( \hat{f}^i(u) + \widetilde{f}^i(w)) \otimes w_i + g^j(u) \otimes u_j\,.
    \]
    Contracting by $u_j$, checking the symmetry condition and looking for the term involving $u^i$ we obtain
    \[
    \iota_{u_j} \hat{f}^i \iota_{w_i} \alpha + \iota_{u_j} g^i \alpha_i = 0\,,
    \]
    so that, by Eq. \eqref{eq:condition_k_iso}, we have $\hat{f}^i = g^i = 0$, finishing the proof.
\end{enumerate}
\end{proof}

\begin{proof} [Proof of Theorem \ref{thm:casuistic_of_representation}] Assume that the representation is neither irreducible, nor decomposable into a direct sum of representations and such that it does not admit an invariant $j$-isotropic subspace.

Then, there exists a proper $G$-invariant subspace $\{0\} \subsetneq W \subsetneq V$ and there are none of these that are $k$-isotropic. Consider
\[
W^{\perp, k} := \{v \in V \colon \iota_v \iota_{w_k} \cdots \iota_{w_1}\varpi = 0 \,, \forall w_1, \dots, w_k \in W \}\,.
\]
Then, $W^{\perp, k}$ is $G$-invariant hence so is $W \cap W^{\perp, k}$. However, by construction $W \cap W^{\perp, k}$ is $k$-isotropic, so we necessarily have $W \cap W^{\perp, k} = \{0\}$. If we have $W^{\perp, k} \neq \{0\}$, we may consider $\widetilde W := W \oplus W^{\perp, k}$ and iterate the argument. Eventually, at a finite step, since the representation does not decompose as a sum of representations, we will obtain an invariant $W$ such that $W^{\perp, k} = \{0\}$, and by applying Lemma \ref{lemma:semi_finite_condition}, we finish the proof.
\end{proof}

We now give some results regarding multisymplectic representations with an invariant $W$ satisfying $W^{\perp, k} = \{0\}$.

\begin{proposition}
    \label{prop:injective_mapping} Let $\rho \colon G \rightarrow \GL(V)$ be a multisymplectic representation and suppose that there is an invariant subspace $W$ with $W^{\perp, k} = \{0\}$. Then, the map 
    \[
    \Phi \colon \mathfrak{g} \rightarrow \rho_W(\mathfrak{g})
    \]
    given by restriction is injective.
\end{proposition}
\begin{proof} Let $\{w_i, u_j\}$ be a basis of $V$ such that $W = \langle w_i\rangle$. Let 
\[
T = A^i_j w^j \otimes w_i + u^j \otimes \left(B_j^i u_i + C_j^i w_i\right) \in \mathfrak{g} \subseteq \mathfrak{g}_\varpi\,.
\]
Assume $\Phi(\xi) =0$, so that $A^i_j = 0$. Suppose that in this basis
\[
\varpi  = \alpha + u^i \wedge \alpha_i + \text{terms involving $u^i \wedge u^j$}
\]
and that $\iota_T \varpi = 0$. Then, we have 
\[
0 = u^j \wedge \left( C^i_j \iota_{w_i} \alpha + B^i_j \alpha_i\right) + \text{terms involving $u^i \wedge u^j$}\,.
\]
In particular we have $C^i_j \iota_{w_i} \alpha + B^i_j \alpha_i = 0$, for every $j$. This implies that the vectors $C^i_j w_i + B^i_j u_i \in W^{\perp, k} = \{0\}$, so that $C^i_j = B^i_j = 0$ and the result follows. 
\end{proof}

 Finally, it will be of interest to have an explicit description of the Lie algebra $\mathfrak{g}_\varpi$ in the presence of a $1$-isotropic invariant subspace:

\begin{proposition}
\label{prop:explicit_description_of_alegbra_1isotropic}
Let $\varpi \in \bigwedge^k V^\ast$ be of $1$-isotropic type (so that there is a $G_\varpi$-invariant and $1$-isotropic subspace $W$). Suppose that we can find a $k$-isotropic complement $L \oplus W = V$ (not necessarily $G_\varpi$-invariant). Define the following elements:
\begin{enumerate}[\rm (i)]
    \item We first define\[
    S := \{\alpha \in \bigwedge^{k-1} L^\ast \colon \alpha = \iota_v \varpi \,, \text{ for some } v \in W \}.
        \]
        and let \[
    \mathfrak{g}_S := \{T \in L^\ast \otimes L \colon \iota_T S \subseteq S\}
    \]
denote the Lie algebra of symmetries of $S$.
\item Let
$
\mathcal{A}_\varpi := \{\alpha \in L^\ast \otimes S \colon \partial \alpha = 0\}\,.
$ Here, $\partial$ denotes the Spencer differential.
\end{enumerate}
Then, by interpreting $\mathcal{A}_\varpi$ as an abelian Lie algebra, there is an isomorphism
\[
\mathfrak{g}_\varpi \cong \mathfrak{g}_S \ltimes \mathcal{A}_\varpi\,,
\]
where the action of $\mathfrak{g}_S$ on $\mathcal{A}_\varpi$ is given by contraction.
\end{proposition}
\begin{proof} Denote by $\{l_i, p_j\}$ an adapted basis to the decomposition $V = L \oplus W$. Notice that since $W$ is $1$-isotropic and $L$ is $k$-isotropic we have that there are $(k-1)$-forms $\alpha_j \in \bigwedge^{k-1} L^\ast$ satisfying
\[
\varpi = p^j \wedge \alpha_j\,.
\]
Then, since $W$ is invariant, the most general element $T \in  \mathfrak{g}_\varpi$ is of the form
\[
T = \widetilde T + \gamma^j \otimes p_j + A^j_{i} p^i \otimes p_j\,,
\]
where $\gamma^j \in L^\ast$ and $\widetilde T \in L^\ast \otimes L$. We have that $\iota_T \varpi = 0$ is equivalent to
\[
\iota_{\widetilde T} \alpha_j = - A^i_j \alpha_i \, \qquad \gamma^j \wedge \alpha_j = 0\,.
\]
$\widetilde T$ satisfying the first condition is equivalent to $\widetilde T \in \mathfrak{g}_S$, while the family $\gamma^j$ satisfying the latter is equivalent to $\gamma^j \otimes \alpha_j \in \mathcal{A}_\varpi$. Let us define the map
\[
\Phi \colon \mathfrak{g}_\varpi \rightarrow \mathfrak{g}_S \oplus \mathcal{A}_\varpi \,, \qquad \Phi(T) = (\widetilde T,  \gamma^j \otimes \alpha_j)\,.
\]
Then $\Phi$ is a linear isomorphism. It is evidently surjective, and injectivity follows by non-degeneracy of $\varpi$. Now, let us compute how it relates to the commutator:
\begin{align*}
    \Phi([T_1, T_2]) &= \Phi\left([\widetilde T_1, \widetilde T_2] +  \iota_{\widetilde T_1} \widetilde \gamma^j \otimes p_j - \iota_{\widetilde T_2} \gamma^j \otimes p_j - \widetilde A^i_j \gamma^j \otimes p_i + A^i_j \widetilde \gamma^j \otimes p_i + (A^j_k \widetilde A^k_i - \widetilde A^j_k A^k_i) p^i \otimes p_j\right)\\
    &= \left([\widetilde T_1, \widetilde T_2],\iota_{\widetilde T_1} \widetilde \gamma^j \otimes \alpha_j - \iota_{\widetilde T_2} \gamma^j \otimes \alpha_j - A^i_j \widetilde \gamma^j \otimes \alpha_i + \widetilde A^i_j  \gamma^j\otimes \alpha_i  \right)\\
    &= \left([\widetilde T_1, \widetilde T_2] , \iota_{\widetilde T_1} \widetilde \gamma^j \otimes \alpha_j - \iota_{\widetilde T_2} \gamma^j \otimes \alpha_j + \widetilde \gamma^j \otimes \iota_{\widetilde T_1} \alpha_j - \gamma^j \otimes \iota_{\widetilde T_2} \alpha_j\right)\,,
\end{align*}
so that we obtain the semi-direct product structure, with the representation of $\mathfrak{g}_S$ on $\mathcal{A}_\varpi$ given by
\[
T \cdot (\gamma \otimes \alpha) = (\iota_T \gamma) \otimes \alpha + \gamma \otimes \iota_T \alpha = \iota_T(\gamma \otimes \alpha)\,.
\]
\end{proof}

The above result may be used to give an explicit computation of the extensions in a plethora of scenarios:

\begin{corollary}
\label{cor:extensions_in_1_isotropic_case}
In the hypotheses of the Proposition before, we have that
\[
\mathfrak{g}_\varpi^{(j)} \cong \mathfrak{g}_S^{(j)} \oplus \{\alpha \in S^{(j+1)}(L^\ast) \otimes S \colon \partial \alpha = 0\}\,.
\]
\end{corollary}

\subsection{Flatness theorems}
\label{subsection:flatness_results}

In this section, we study flatness results of $G_\varpi$-structures. Equivalently, conditions on a form of constant linear type to be of constant coefficients.

In the general context of the theory of $G$-structures, one could state the following conjecture
\[
\text{ \textbf{Conjecture:} \qquad Every formally flat $G$-structure is flat}\,.
\]
If this conjecture holds, the differential conditions on a form $\varpi$ to admit a chart in which it has constant coefficients would be given by vanishing of the structure tensors taking values in the bundle with standard fiber $H^{2, j}(\mathfrak{g}_\varpi)$.

As far as we are aware, in the literature there is a debate regarding whether the previous conjecture is a proved result or not (see \cite{Cattafi2020AlmostStructures}). We are aware of several papers in which the authors claim they prove the previous conjecture to be true, namely \cite{Pollack1974, ButtinMolino1974, Goldschmidt1978}, and in the second one they only refer to the rest as proof attempts. The difficult step is the following: If $G$ does not act irreducibly, we know that there are induced $\rho_W(G)$ and $\rho_{V/W}(G)$-structures on the leaves and on the space of leaves, respectively. One may try an iterative argument to proving the previous conjecture and use that there exist coordinates on the leaves and on the leaf space (the local quotient by the foliation) that make their respective structures flat. These coordinates can be understood as a coordinate chart on $M$. However, they do not necessarily make the original $G$-structure flat. In general, it is hard to guarantee that there is a modification of the coordinates that proves flatness.

However, the case of structures induced by multisymplectic forms, this last step may be amended through the \emph{Moser trick}. Let $\varpi \in \bigwedge^{k+1} V^\ast$  and let $M$ be endowed with a $G_\varpi$-structure. Let $S := \{\iota_v \varpi \colon v \in V\} \subseteq \bigwedge^k V^\ast$ and define:
\[
G_S := \{g \in \GL(V) : g^\ast S = S\}\,.
\]
Notice that we have an inclusion $G_\varpi \subseteq G_S$. In fact, on $M$ there is an induced $G_S$-structure (which is the bundle $\{\iota_v \omega \colon v \in \T M\} \subseteq \bigwedge^k \T^\ast M$).

\begin{proposition}
\label{prop:Moser_trick}
A closed form $\omega \in \Omega^{k+1}(M)$ with constant linear type admits a coordinate chart in which it has constant coefficients if and only if there is a coordinate chart on $M$ in which the bundle \[
\{\iota_v \omega \colon v \in \T M \} \subseteq \bigwedge^k \T^\ast M
\]
is generated by forms of constant coefficients.
\end{proposition}

\begin{proof} If $\omega$ admits a chart in which it has constant coefficients, it follows that the bundle $\{\iota_v \omega \colon v \in \T M\}$ is generated by forms with constant coefficients. Conversely, suppose that the previous bundle is generated by forms of constant coefficients, say $\alpha_1, \dots, \alpha_n$ in some chart $(x^j)$ around a particular point in $M$, which we denote by $p$. Without loss of generality, we may assume that $\omega|_p = \varpi$. We apply the Moser trick now. By applying the Poincaré Lemma in the previous coordinates, we can find $k$-forms $\theta$, $\vartheta$ which take values in the bundle generated by $\alpha_i$ and such that $\omega = \d \theta$ and $\varpi = \d \vartheta$. Let us define \[
\omega_t := (1-t) \omega + t \varpi\,.
\]
Then, $\omega_t$ is multisymplectic in some neighborhood of $p$ and, more importantly, it satisfies that
\[
\{\iota_v \omega_t \colon v \in \T M\} = \langle \alpha_i \rangle\,,
\]
for $t \in [0,1]$ if this neighborhood is made smaller. It follows that there exists a unique time-dependent vector field such that
\[
\iota_{X_t} \omega_t = \theta - \vartheta\,.
\]
Now we employ the Moser trick \cite{mos1965trans.am.math.soc.} to show that the time $1$ flow of the previous vector field gives a diffeomorphism that flattens $\omega$. Indeed, let $\psi_t$ denote this flow. Since there is an equilibrium at $p$, by reducing the neighborhood further we may suppose that this flow is defined for $t \in [0,1]$. Now we compute the derivative:
\begin{align*}
    \dv{t} \left( \psi_t^\ast \omega_t\right) &= \psi_t^\ast \left(\pounds_{X_t} \omega_t + \dot \omega_t\right)\\
    &= \psi^\ast_t \left( \d (\theta - \vartheta) - \omega + \varpi\right) = 0\,,
\end{align*}
which shows that $\psi_1$ satisfies $\psi_1^\ast \varpi = \omega$, and finishes the proof.
\end{proof}

\begin{remark}[Moser's trick adapted to a distribution]
\label{remark:Moser_trick_adapted}
In some scenarios, it is sufficient to look for vector fields taking values in a particular distribution. Exactly the same argument follows when we are capable of justifying that the potential may be taken to have values in the bundle $\{\iota_v \omega \colon v \in \mathcal{D}\}$, where $\mathcal{D} \subseteq \T M$ is the distribution. 
\end{remark}



Hence, to flatten a multisymplectic form of constant linear type $\omega \in \Omega^{k+1}(M)$ it is enough to flatten its image. We now study the obstructions to achieve this in the casuistic given before: Let us first give a sufficient condition for finding flat coordinates in the $j$-isotropic case:

\begin{theorem}
\label{thm:sufficient_condition_flatness_general}
Let $\varpi \in \bigwedge^{k+1} V^\ast$ be of $j$-isotropic type, with $j$-isotropic subspace $W$ such that it admits a (not necessarily invariant) $(k+1-j)$-isotropic complement $L$. Let $\pi  \colon B \rightarrow M$ be a formally flat $G_\varpi$-structure on $M$. Suppose that there are coordinates $(w^i, x^\mu)$ satisfying:
\begin{enumerate}[\rm (i)]
    \item They are adapted to the foliation induced by $W$, namely
    \[
B[W] = \left \langle \pdv{w^i}\right \rangle
\]
    \item They make flat the geometry induced in the leaves and in the quotient.
    \item The distribution $\left\langle \pdv{x^\mu}\right\rangle$ is $(k+1-j)$-isotropic.
\end{enumerate}
Then, the $G_\varpi$-structure is flat.
\end{theorem}

\begin{proof} Let $(w^i, x^\mu)$ be coordinates such as above. Then, the form $\omega$ induced by $\varpi$ reads as
\[
\omega= \d w^{i_1} \wedge \cdots \wedge \d w^{i_j} \wedge \alpha_{i_1 \dots i_j}\,,
\]
where $\alpha_{i_1 \dots i_j}$ are only in $\d x^\mu$. Then, let us observe that, because $(x^\mu)$ are coordinates adapted to the geometry transverse to the leaves, we have that there are forms $\widetilde \alpha_{i_1, \dots, i_j}$ of constant coefficients such that
\[
\langle \alpha_{i_1, \dots, i_j}\rangle = \langle \widetilde  \alpha_{i_1, \dots, i_j}\rangle\,.
\]
In particular, these coordinates make the distribution $\{\iota_v \omega \colon v \in \T M\}$ flat, and so, by Proposition \ref{prop:Moser_trick}, there are flat coordinates for $\omega$ (equivalently, the $G_\varpi$-structure).
\end{proof}

In general, it is hard to find the explicit coordinates that make the form as above. In fact, we can give some necessary conditions for these coordinates. This will result useful in Section \ref{section:obstruction_of_fixed_order}, where we will build certain linear types $\varpi_j$ with obstructions to integrability of order $j$.

\begin{proposition}
\label{prop_sufficiency_condition_solving_flatnes}
Let $\varpi \in \bigwedge^{k+1} V^\ast$ be of $j$-isotropic type such that it admits a (not necessarily invariant) $(k+1-j)$-isotropic complement. Let $M$ be a manifold endowed with a formally flat $G_{\varpi}$-structure. Then, locally around any point we can find a sequence $\alpha_{i_1, \dots, i_m} \in \Omega^{k+1-m}(M)$ of forms, for $i_1 < \cdots < i_m$ and $i_l \in \{1, \dots, \dim W\}$, $m \leq j$, such that if the $G_\varpi$-structure is integrable then the following hold:
\begin{enumerate}[\rm (i)]
    \item The leaf geometry and geometry transversal to the leaves is integrable.
    \item The following PDE on functions $f^i\in C^\infty(M)$ for $i \in \{1, \dots, \dim W\}$ admits at least one solution:
    \[
    \sum_{l = 0}^{j - m} \binom{m+l}{m}\d f^{i_1} \wedge \cdots \wedge \d f^{i_l} \wedge \alpha_{i_1 \cdots i_{l+m+1}} = 0\,.
    \]
\end{enumerate}
\end{proposition}

\begin{proof} Let $(w^i, x^\mu)$ be fibered coordinates on $M$ adapted to the foliation induced by the invariant $G_\varpi$ subspace $W$, namely such that
\[
B[W] = \left\langle \pdv{w^i}\right\rangle\,.
\]
Then, since $W$ is $j$-isotropic, there is a sequence of forms $\alpha_{i_1, \dots, i_m} \in \Omega^{k+1-m}(M)$ only depending on $\d x^i$ such that
\[
\omega = \alpha + \d w^i \wedge \alpha_i + \cdots + \d w^{i_1}  \wedge \cdots \wedge \d w^{i_j}\wedge \alpha_{i_1, \dots, i_j}\,.
\]
Let us first assume that the structure is flat: Then, in particular, there is a transversal submanifold to the foliation, say $L$, which is $(k+1-j)$-isotropic. Taking $\widetilde L$ to be the leaf corresponding to the coordinates $w^i = \text{constant}$ and by repeating the argument above, we obtain a diffeomorphism $f \colon \widetilde L \rightarrow L$ which, again, in coordinates reads as $w^i = f^i(x^\mu)$. Since $\widetilde L$ is $j$-isotropic, if we write the form in the coordinates $\widetilde w^i = w^i + f^i$, we obtain conditions
\begin{align*}
    \omega =& \alpha + \d f^i \wedge \alpha_i + \d f^{i_1} \wedge \d f^{i_2} \wedge \alpha_{i_1 i_2} + \cdots +  \d f^{i_1} \wedge \cdots \wedge \d f^{i_j} \wedge \alpha_{i_1 \dots i_j}\\
    & + \d\widetilde  w^i \wedge \left(  \alpha_i  + 2 \d f^{i_2} \wedge \alpha_{i i_2} + \cdots + j \d f^{i_2} \wedge \cdots \wedge \d f^{i_{j}} \wedge \alpha_{i i_2, \dots, i_j}\right)\\
    & \qquad  \vdots \\
    & +\d \widetilde w^{i_1} \wedge \cdots \wedge \d \widetilde w^{i_{j-1}} \wedge \left(\alpha_{i_2 \dots i_{j}} + j \d f^{i_1} \wedge \alpha_{i_1,\dots, i_{j}} \right)\\
    & + \d \widetilde w^{i_1} \wedge \cdots \wedge \d \widetilde w^{i_j} \wedge \alpha_{i_1, \dots, i_{j}}\,,
\end{align*}
which on $L$ forces every term to vanish, except the last one, finishing the proof.
\end{proof}

However, when $j = 1$, one is always capable of finding coordinates such as in Theorem \ref{thm:sufficient_condition_flatness_general} when the geometry of the leaves and quotient is flat.

\begin{theorem}
\label{thm:Darboux_1_isotropic}
Let $\varpi \in \bigwedge^{k+1} V^\ast$ be a form of $1$-isotropic type. Then, a formally flat $G_\varpi$-structure is flat if and only if the induced geometries on the leaves and quotient are flat.
\end{theorem}

Before proving this result, we will prove the following Lemma.

\begin{lemma}
\label{lemma:flatness_1_isotropic}
Let $\varpi \in \bigwedge^{k+1} V^\ast$ be of $1$-isotropic type with invariant $1$-isotropic subspace $W$ admitting a (not necessarily invariant) $k$-isotropic complement. Let $\pi \colon B \rightarrow M$ be a formally flat $G_\varpi$-structure on a manifold $M$. Then, locally around every point there are certain $k$-forms $\alpha_i$, for $i \in \{1, \dots, \dim W\}$ of constant coefficients (in a particular chart) and a $(k+1)$-form $\alpha$ (not necessarily of constant coefficients) such that the $G_\varpi$-structure is flat if and only if
\begin{enumerate}[\rm (i)]
    \item The leaf geometry and the geometry transversal to the leaves are flat,
    \item The PDE on maps $f \in C^\infty(M)$ constant on the leaves given by
    \[
    \alpha = -\d f^i \wedge \alpha_i
    \]
    has at least one solution.
\end{enumerate}
\end{lemma}

\begin{proof}[Proof of Theorem \ref{thm:Darboux_1_isotropic}] Let $(w^i, x^\mu)$ denote coordinates adapted to the foliation induced by the distribution $B[W]$, in such a way that these coordinates make the induced geometric structures on the leaves and quotient flat. Then, there are forms $ \alpha$ and $\widetilde \alpha_i$, only depending on $\d x^i$ such that, if we denote by $\omega$ the form induced by the structure, we have
\[
\omega =  \alpha + \d w^i \wedge \widetilde \alpha_i\,.
\]
 Since, by Proposition \ref{prop:explicit_description_of_alegbra_1isotropic} the geometry on the quotient is induced by the bundle $\langle \widetilde \alpha_i \rangle$, and we are assuming that the coordinates $(x^\mu)$ make this geometry flat, we may assume that there are certain forms $\alpha_i$ of constant coefficients in this chart in such a way that
    $\widetilde \alpha_j = A^i_j \alpha_i$, for certain invertible matrix $A^i_j$. Hence,
    \[
    \omega = \alpha + A^i_j \d w^j \wedge \alpha_i\,,
    \]
    where each $\alpha_i$ is of constant coefficients. Since $\omega$ is closed (the $G_\varpi$ structure being formally flat), the functions $A^i_j$ satisfy $\pdv{A^i_j}{w^k} \d w^k \wedge \d w^j = 0$ so that, in particular, $A^i_j = \pdv{g^i}{w^j}$. Using these functions as a coordinate change for the $w^i$ we obtain that the form may be locally expressed as $\omega = \alpha + \d w^i \wedge \alpha_i$, where the $\alpha_i$ are of constant coefficients.

    If there was a chart in which $\omega$ has constant coefficients, there would exist a $k$-Lagrangian submanifold transverse to the leaves, say $L$. Now, let us take a particular leaf $\widetilde L$ induced by the chosen coordinates (namely corresponding to $w^i = \text{constant}$). Then both projections to the (locally defined) quotient manifold
    \[
    p|_L \colon L \rightarrow M/B[W] \,, \qquad p|_{\widetilde L} \colon \widetilde L \rightarrow M/ B[W]
    \]
    define local diffeomorphisms, which by making these manifolds smaller may be assumed to be diffeomorphisms. Let
    \[
    f := (p|_L)^{-1} \circ p|_{\widetilde L} \colon \widetilde L \rightarrow L
    \]
    denote the diffeomorphism defined by identifying both submanifolds with the quotient. In coordinates, this diffeomorphism corresponds to some family of functions
    \[
    w^i =  f^i(x^\mu)\,.
    \]
    Since $L$ is $k$-Lagrangian, we have $\omega |_L = 0$ and, in particular this map satisfies
    \[
    \alpha + \d f^i \wedge \alpha_i = 0\,.
    \]

    Conversely, if the PDE above admits a solution, we can build the transversal $k$-Lagrangian submanifold. Now we may apply Moser's trick adapted to the distribution $B[W]$ simply by noticing that the relative Poincaré Lemma guarantees the existence of a potential taking values in $\langle \alpha_i\rangle$, from which the result follows.
\end{proof}

\begin{proof}[Proof of Theorem \ref{thm:Darboux_1_isotropic}] In the coordinates given by Lemma \ref{lemma:flatness_1_isotropic}, it only remains to show that the PDE $\alpha = - \d f^i \wedge \alpha_i$ admits one solution. In the same local chart, this equation is equivalently expressed as
\[
P[D] \cdot f = \alpha\,,
\]
where $P$ is a matrix of polynomials of degree $1$, $D = {(\pdv{x_i})}_i$ is the vector of partial derivatives, and $\alpha \in C^\infty(M)^N$ is some collection of functions (representing the coefficients of the form $\alpha$). Since $P$ is a linear differential operator of constant coefficients, it has some associated compatibility operator $Q$ (given by Theorem \ref{thm:Hormander} in the Appendix) that guarantees that if $Q[D] \cdot \alpha = 0$, there exists a smooth local solution $f$. We will show that this is indeed the case.

Let $x \in M$, we will show $Q[D] \cdot \alpha = 0$. The operator $Q$ is of certain finite order, say $k$. Then, there are certain coordinates on $M$ that make the structure $(k+1)$-flat at $x$. These coordinates are given by certain choices of functions $g^i$ that solve
\[
\d g ^i \wedge \alpha_i = \alpha \qquad \text{at }x \text{ up to order $(k+1)$}\,.
\]
Then, $\alpha - \d g^i \wedge \alpha_i = O(k+1)$ at $x$. Hence:
\[
0 = Q[D] \cdot (\alpha - \d g^i \wedge \alpha_i) |_x= Q[D] \cdot \alpha|_x\,,
\]
since $\d g^i \wedge \alpha_i = P[D] \cdot g$, and $Q \cdot P = 0$.
\end{proof}

These last results, together with the alternative given in Theorem \ref{thm:casuistic_of_representation} may be used in a plethora of scenarios to prove a Darboux theorem by following a particular strategy:

\begin{enumerate}
    \item If $G_\varpi$ acts irreducibly. Then, by Theorem \ref{thm:irreducible_integrability}, formally flat already implies flat.
    \item If not, we have three cases, given by Theorem \ref{thm:casuistic_of_representation}:
    \begin{enumerate}[\rm (a)]
        \item If the action decomposes as a direct sum of representations $V = W_1 \oplus W_2$ then, being flat is equivalent to each of the induced structures on each leaf being flat. In some cases (mainly the ones presented in this text), the induced representations may be realized as
        \[
        \rho_{W_1}(G) = G_{\varpi_1} \,, \qquad \rho_{W_2}(G) = G_{\varpi_2}\,,
        \]
        for certain (not necessarily $k$) forms $\varpi_i \in \bigwedge W_i^\ast$. We then may iterate the argument and, at a finite step, arrive at a trivial case (e.g. a symplectic form, or volume element), which are known to satisfy that formally flat implies flat.
        \item If there is an invariant $j$-isotropic space, not much can be said, in general. However, in some scenarios (those involving classical fields, for instance), the subspace will actually be $1$-isotropic so that, in light of Theorem \ref{thm:Darboux_1_isotropic}, it is enough to find coordinates that make the $\rho_{V/W}(G_\varpi)$-structure flat.
        \item Finally, if there is an invariant subspace $W$ with $W^{\perp, k} = \{0\}$, the discussion above happens in the other way around, namely the geometry on the leaf-space is of finite type. 
    \end{enumerate}
\end{enumerate}

It should be noted that this strategy (when successfully applied) only shows that the tensorial conditions on a $G_\varpi$-structure to be flat exist (being those that guarantee it being formally flat), and reduces it to a matter of linear algebra.

In order to show the usefulness of this procedure we apply it to some examples:

\begin{example}[Multicontangent bundles]
\label{example:Multicotangent_bundle}
Let $L$ denote an arbitrary vector space and let $V := L \oplus \bigwedge^{k-1} L^\ast$, for $k \geq 3$. On $V$ there is a canonical $k$-form
\[
\varpi((l_1, \alpha_1), \dots, (l_{k}, \alpha_k)) := \sum_{j = 1}^{k} (-1)^{j-1} \alpha_j (l_1, \dots, \widehat{l}_j, \dots, l_k)\,.
\]
It is well known that the action of $G_\varpi$ on $V$ has an invariant subspace, which is
\[
W =\{0\} \oplus \bigwedge^{k-1}L^\ast\,.
\]
This subspace is actually $1$-isotropic and the structure induced on the quotient is a $\GL(L)$-structure. Hence, by Theorem \ref{thm:Darboux_1_isotropic}, any formally flat $G_\varpi$-structure is actually flat.
\end{example}

\begin{example}[Product type] Let $L_1$ and $L_2$ denote vector spaces of arbitrary dimension, and $\varpi_1 \in \bigwedge^k L_1^\ast$ and $\varpi_2 \in \bigwedge^k L_2^\ast$ denote multisymplectic $k$-forms, for $k\geq 3$. Define on $V = L_1 \oplus L_2$ the multisymplectic form $\varpi:= \varpi_1 + \varpi_2$. Then, a simple computation shows that the action of $G_\varpi$ on $V$ actually decomposes as the direct sum
\[
V = L_1 \oplus L_2\,.
\]
Furthermore, the induced representations on $L_1$ and $L_2$ are actually given by $G_{\varpi_1}$ and $G_{\varpi_2}$. Hence, we conclude that a manifold $M$ endowed with a multisymplectic form $\omega$ of linear type $\varpi$ which is formally flat is flat if and only if any pair of manifolds $(M_1, \omega_1)$ and $(M_2, \omega_2)$ endowed with multisymplectic forms of linear type $\varpi_1$ and $\varpi_2$ which are formally flat, are actually flat. A particular example, referred to in the literature as \textit{product type} (see \cite{Ryvkin-Darboux-type-theorems-2025}) is when $\varpi_1$ and $\varpi_2$ are volume forms, so that $G_{\varpi_i} = \operatorname{SL}(n)$, whose action is irreducible. Hence, by Theorem \ref{thm:irreducible_integrability}, any such structure which is formally flat, is flat. It should be noted that this follows simply by employing the Moser trick.
\end{example}


\begin{example}[Density valued symplectic forms \cite{leonid}]
\label{example:Density-valued-forms}
Let $V$ be a vector space of dimension $k +2n$, together with a basis $\{e_1, \dots, e_k, f_1, \dots, f_{2n}\}$. Define
\[
\varpi := e^1 \wedge \cdots \wedge e^k \wedge \left( f^1 \wedge f^2 + \cdots + f^{2n-1} \wedge f^{2n}\right) \in \bigwedge^{k+2} V^\ast\,.
\]
It can be shown that the action of $G_\varpi$ on $V$ has an invariant $2$-isotropic space,
\[
W := \{f_1, \dots, f_n\}\,.
\]
In this scenario, both the leaf geometry, which is the conformally symplectic group, and the transversal geometry, which is $\GL(r)$, are of infinite type. It is well known, however, that formal flatness of these two geometries implies flatness. Hence it is enough to justify that for any formally flat $G_\varpi$-structure, we can find coordinates in the hypotheses of Theorem \ref{thm:sufficient_condition_flatness_general}. However, simply by dimensional considerations, we have that any set of coordinates (in particular the coordinates adapted to each geometry) induces a $2$-isotropic and $(k-1)$-isotropic decomposition of $\T M$. Applying the result, we conclude that any formally flat $G_\varpi$-structure is actually flat.
\end{example}

\section{Forms with obstructions to integrability of a fixed order}

\label{section:obstruction_of_fixed_order}

It should be noted that flatness theorems for multisymplectic structures are usually of \emph{first order} (see \cite{Ryvkin-Darboux-type-theorems-2025, glrr2024revistadelarealacademiadecienciasexactasfisicasynaturales.seriea.matematicas}), namely that, in local coordinates, the conditions on a form $\omega \in \Omega^k(M)$ of constant linear type $\varpi \in \bigwedge^k V^\ast$ to have constant coefficients are of first order in the coefficients of the form in some chart. This motivates the following question:
\[
\text{Is there a bound for the order appearing on a `Darboux theorem' for a multisymplectic form?}
\]
In terms of the Spencer cohomology, this question can be expressed as the existence of an index, $j_0$, for which $H^{2, j}(\mathfrak{g}_\varpi) = \{0\}$, for all $j \geq j_0$ and for all $\varpi \in \bigwedge^{k} V^\ast$. As one important application of the theory of $G$-structures, we show that the answer is negative, namely for every $j\geq 1$, we build a $3$-form $\varpi_{j}$ on a vector space of dimension $2j+5$ that satisfies
\[
H^{2,j}(\mathfrak{g}_\varpi) \neq \{0\}\,.
\]
In particular, we obtain linear types of differential forms that have higher order obstructions to integrability, thus providing counterexamples for a Darboux theorem of first (or bounded) order. As introduced above, we employ Cartan's tableaux method, and the main constructions are based on Corollary \ref{corollary:injective_mapping_homology}, namely, we will define a family of tableaux $\mathcal{A}_{j}$ satisfying $H^{2,j}(\mathcal{A}_{j}) \neq \{0\}$. Then, we will identify forms $\varpi_{j}$ in such a way that $\mathcal{A}_{j}$ arises as a retraction of $\mathfrak{g}_{\varpi_j}$ in a particular way so that it satisfies the conditions of Corollary \ref{corollary:injective_mapping_homology}. Then, it will be clear that
\[
\{0\} \neq H^{2,j}(\mathcal{A}_{j}) \hookrightarrow H^{2,j}(\mathfrak{g}_{\varpi_j}) \,,
\]
so that $H^{2,j}(\mathfrak{g}_\varpi)$ is necessarily nonzero. Finally, we identify explicit $\omega_j \in \Omega^3(\mathbb{R}^{5 +2j})$ which has constant linear type $\varpi_j$ and such that the structure tensors satisfy $c_0 = \cdots = c_{j-1} = 0$, but $c_j \neq 0$.

\subsection{The definition and cohomology of \texorpdfstring{$\mathcal{A}_{j}$}{}}

Let us define the tableaux $\mathcal{A}_{j}$ as follows: Let $V = \langle y_1, y_2\rangle$ be a $2$-dimensional vector space and $W = \langle p_0, \dots, p_j\rangle$ be a $(j+1)$-dimensional vector space. Then, let us define
\[
\mathcal{A}_j := \left \langle y^1 \otimes p_0 + y^2 \otimes p_1, \dots, y^1 \otimes p_{i} + y^2 \otimes p_{i+1}, \dots, y^1 \otimes p_{j-1} + y^2 \otimes p_j \right \rangle\,.
\]
In terms of matrices (or their transpose for the presentation), we can also write:
\[
\mathcal{A}_j^{T} = \left\{
\begin{pmatrix}
    a_0 & a_1 & \cdots & a_j & 0\\
    0   & a_0 & \cdots & a_{j-1} & a_j
\end{pmatrix} \colon a_i \in \mathbb{R}
\right\}\,.
\]

Then we have the following
\begin{theorem}
\label{theorem:cohomology_in_degree_2}
The Spencer cohomology of $\mathcal{A}_{j}$ is given by
\[
H^{l', j'}(\mathcal{A}_{j}) = \begin{cases}
0 \qquad \text{if} \qquad (l', j') \neq (2, j)\\
\mathbb{R} \qquad \text{if} \qquad (l',j') = (2, j)
\end{cases}\,.
\]
\end{theorem}
Notice that these diagrams embed onto one another. In fact, the proof will follow inductively, first computing the prolongations, then their cocycles, and finally their cohomology. More specifically, these Cartan tableaux will be such that $\mathcal{A}_{j}^{(j-1)} \cong \mathbb{R}$ and such that $\mathcal{A}_{j-1}^{(i)} = \{0\}$, for $i \geq j$.

Let us begin, as stated, by studying its prolongations. We have a natural element in $\mathcal{A}_j^{(j-1)}$, for every $j$, which we denote by $\alpha_j \in S^{j}(V^\ast) \otimes W$. In the basis above, it is defined as
\[
\alpha_j = \frac{(y^{1})^{j}}{j!} \otimes p_0 + \cdots + \frac{(y^1)^{j-k}}{(j-k)!} \frac{(y^2)^{k}}{k!} \otimes p_k + \cdots + \frac{(y^2)^{j}}{j!} \otimes p_j\,.
\]

Also notice that for $0 \leq A \leq j-i$ we have a natural embedding
\[
\mathcal{A}_i \xrightarrow{\mathfrak{i}_A} \mathcal{A}_j\,.
\]
This may be employed to define iteratively elements in $\mathcal{A}_j^{(i-1)}$. Indeed, if $\alpha_i \in \mathcal{A}_i^{(i-1)}$ denotes the element defined before we define
\[
\alpha_{i}^A := \mathfrak{i}_A (\alpha_i)\,.
\]
Algebraically, we have
\[
\alpha^A_i = \frac{(y^1)^i}{i!} \otimes p_A + \cdots + \frac{(y^1)^{i-k}}{(i-k)!} \frac{(y^2)^{k}}{k!} \otimes p_{A + 1 + k} + \cdots + \frac{(y^2)^{i}}{i!} \otimes p_{A+1+i}\,.
\]
\begin{remark}
\label{remark:computation_of_partial_alpha}
It follows immediately that we have
\[
\partial \alpha^A_i = y^1 \otimes \alpha^A_{i-1} + y^2 \otimes \alpha^{A+1}_{i-1}\,.
\]
\end{remark}
\begin{lemma}
\label{lemma:generators_prolongations}
We have that the prolongations of $\mathcal{A}_j$ are
\begin{enumerate}[\rm (i)]
    \item For $1 \leq i \leq j-1$, we have $\mathcal{A}_j^{(i-1)} = \langle \alpha_i^{0}, \dots, \alpha_i^{j-i} \rangle$.
    \item For $i \geq j$, we have $\mathcal{A}_j^{(i)} = \{0\}$.
\end{enumerate}
\end{lemma}
\begin{proof} Indeed, let $\beta \in \mathcal{A}_j^{(i-1)}$. Then, when written in the natural basis of $S^{i}(V^\ast) \otimes W$ given above, namely $\{(y^1)^{k}(y^2)^{i-k} \otimes p_i\}$, there must be a first non-vanishing coefficient accompanying some $p_i$, say for the index $A$. Then, it is clear (as it is the first nonvanishing coefficient) that the monomial accompanying $p_A$ is necessarily, up to a constant, $(y^1)^i$. Suppose that this constant was $\frac{C^A}{i!}$. Then,
\[
\beta - C^A \alpha_i^A \in \mathcal{A}_j^{(i)}\,
\]
lies in the prolongation and has no monomial involving only $y^1$ and $p_A$. Iterating the argument up until $A \leq j-i$ we arrive at an element
\[
\beta - \sum_{A = 0}^{j-i} C^A \alpha_i^A \in \mathcal{A}_{j}^{(i)}
\]
which is such that there is no monomial in $y^1$ and $p_A$, for $0 \leq A \leq j-i$. It only remains to check that this implies that this element vanishes. Indeed, if it had a nonvanishing coefficient on, say $y^1 \otimes p_B$, for $B > j-i$, it would be necessarily proportional to $(y^1)^i \otimes p_A$. However, there is no such element in the prolongation, as a simple computation shows.
\end{proof}

Now we compute its cocycles:
\begin{lemma}
\label{lemma:Computation_of_cocycles}
There are no $0$-cocycles, and the $1$-cocycles are
\begin{enumerate}
    \item On degree $i < j$, we have
    \[
    \ker \partial = \langle \partial \alpha^A_{i+1}\,, A = 0, \dots, j-(i+1)\rangle \subseteq V^\ast \otimes \mathcal{A}_j^{(i-1)}\,.
    \]
    \item On degree $i = j$, we have $\ker \partial = \{0\} \subseteq V^\ast \otimes \mathcal{A}^{(j-1)}$.
\end{enumerate}
Finally, the $2$-cocycles on degree $1\leq i \leq j$ are
\[
y^1 \wedge y^2 \otimes \alpha_i^A \,, \qquad 1 \leq A \leq j+1-i\,.
\]
In particular, all $1$-cocycles are $1$-coboundaries.
\end{lemma}
\begin{proof} That the $2$-cocycles are the ones written above is clear, by dimension considerations. The fact that there are no $0$-cocycles is a general fact for cohomologies with homogeneous polynomial coefficients. Now, for the case of $1$-cocycles, let $\beta \in V^\ast \otimes \mathcal{A}^{(i-1)}_j$ be such that $\partial \beta = 0$. Then, applying Lemma \ref{lemma:generators_prolongations}, $\beta$ may be written as
\[
\beta = \sum_{A = 0}^{j-i} \gamma_A \otimes \alpha^A_i\,.
\]
Now, from Remark \ref{remark:computation_of_partial_alpha}, we have that
\[
\partial \alpha^A_{i} = y^1 \otimes \alpha^{A}_{i-1} + y^2 \otimes \alpha^{A+1}_{i-1}\,.
\]
Hence, if $\beta$ is a coboundary,
\begin{align*}
    0 &= \partial \beta = \sum_{A= 0}^{j-i} \left( y^1 \wedge \gamma_A \otimes \alpha^A_{i-1} + y^2 \wedge \gamma_A \otimes \alpha^{A+1}_{i-1} \right)\\
    &= y^1 \wedge \gamma_0 \otimes \alpha^0_{i-1} +  \sum_{A = 1}^{j-(i+1)}\left( y^1 \wedge \gamma_A + y^2 \wedge \gamma_{A-1}\right)  \otimes \alpha^A_{i-1} + y^2 \wedge \gamma_{j-i} \otimes \alpha^{j-i}_{i-1}\,,
\end{align*}
(the sum in the middle assumed vanishing when $j = i$) so that we obtain the following conditions on the $1$-forms $\gamma_A$:
\begin{align*}
    &y^1 \wedge \gamma_0= 0\,,\\
    &y^1 \wedge \gamma_A + y^2 \wedge \gamma_{A-1}  = 0\,, \qquad \text{for}\qquad A = 1, \dots, j-i\,,\\
    &y^2 \wedge \gamma_{j-i} = 0\,.
\end{align*}
If $i <j-1$, it is easy to check that the space of solutions to this homogeneous linear system of equations is generated by the following
\begin{enumerate}[\rm (i)]
    \item $\gamma_0 = y^1$, $\gamma_1 = y^2$ and $\gamma_3= \cdots = \gamma_{j+1-i} = 0 $,
    \item $\gamma_0 = \cdots = \gamma_{A-1} = 0$, $\gamma_A = y^1$ and $\gamma_{A+1} = y^2$, $\gamma_{A + 2} = \cdots \gamma_{j-i} = 0$, for $A = 1, \dots, j-(i+1)$,
    \item $\gamma_1 = \cdots = \gamma_{j-1-i} = 0$, $\gamma_{j-(i+1)} = y^1$ and $\gamma_{j-i} = y^2$,
\end{enumerate}
corresponding to the $1$-cocycles
\[
y^1 \otimes \alpha^A_i + y^2 \otimes \alpha ^{A+1}_i \,, \qquad \text{for} \qquad A = 0, \dots, j-(i+1)\,.
\]
These generators simply correspond to the coboundaries $\partial \alpha_{i+1}^{A}$. Now, in degree $i = j$, if $\beta = \gamma \otimes \alpha_j$ we obtain the conditions $y^1 \wedge \gamma = 0$ and $y^2 \wedge \gamma = 0$, so that the unique possibility is $\gamma = 0$, showing that there are no $1$-cocycles on degree $j$, and finishing the proof.
\end{proof}

We are finally ready to give the computation of the cohomology of $\mathcal{A}_j$.

\begin{proof}[Proof of Theorem \ref{theorem:cohomology_in_degree_2}] From dimensional considerations, we clearly have $H^{l, j}(\mathcal{A}_j) = \{0\}$, for $l \geq 3$. By Lemma \ref{lemma:Computation_of_cocycles} it is clear that
\[
H^{0, j}(\mathcal{A}_j) = H^{1, j}(\mathcal{A}_j) = \{0\}\,, \qquad \forall j \geq 1\,.
\]
It only remains to show that $H^{2,j}(\mathcal{A}_j) \cong \mathbb{R}$ and that is generated by the cohomology class of $\alpha_j$. That $y^1 \wedge y^2 \otimes \alpha_j$ defines a nonvanishing cohomology class is clear, as, by Lemma \ref{lemma:generators_prolongations}, we have $\mathcal{A}_j^{(j)} =\{0\}$. Hence, we simply need to show that the cocycles
\[
y^1 \wedge y^2 \otimes \alpha^A_i\,, \qquad 0 \leq A \leq j-i \,, \qquad i < j\,,
\]
are actually coboundaries. In fact, in light of Remark \ref{remark:computation_of_partial_alpha} we have
\[
y^1 \wedge y^2 \otimes \alpha^A_i = \partial \left( y^2 \otimes \alpha^{A}_{i+1}\right)\,,
\]
which finishes the proof.
\end{proof}

\subsection{{Construction of the linear type}}

Here we employ the tableaux built in the previous section to obtain the following

\begin{theorem}
\label{thm:Theorem_counterexamples}
Let $j \geq 1$. Then, there exists a nondegenerate $3$-form $\varpi_j \in  \bigwedge^3 V^\ast$, where $V$ is a real vector space of dimension $\dim V = 2j +5$ such that
\[
H^{2, j}(\mathfrak{g}_{\varpi_j}) \neq 0.
\]
\end{theorem}

Hence, we find a linear type of a $3$-form whose flatness depends \emph{strictly} on conditions of order $j$ on the $G_\varpi$-structure. Indeed, we later find a $3$-form $\omega \in \Omega^3(M)$ on a manifold of dimension $2j +5$ with linear type $\varpi_j$ that is flat up to order $(j-1)$ on a particular point $x_0 \in M$ but not to order $j$ on $x_0$. In particular, we obtain

\begin{corollary} The conditions for the flatness of a multisymplectic structure may be of $j$-th order, for arbitrary $j \geq 1$.
\end{corollary}

The proof of Theorem \ref{thm:Theorem_counterexamples} is constructive, and we will simply build a $3$-form $\varpi_j$ such that the \emph{multisymplectic tableau} $\mathfrak{g}_\varpi \subseteq V^\ast \otimes V$ has $\mathcal{A}_{j}$ (defined in the previous section) as a retraction. In particular, in light of Corollary \ref{corollary:injective_mapping_homology}, we will necessarily have $H^{2, j}(\mathfrak{g}_\varpi) \neq \{0\}$. This construction is based on the following lemma.

\begin{lemma}
\label{lemma:necessary_lemma}
Let $V$ be a $(j+4)$-dimensional vector space, with a basis $(x_0, \dots, x_{j+1}, y_0, y_1)$ for $V$. Let $\alpha_0, \dots \alpha_j$ be $2$-forms given in this basis as
\[
\alpha^i := y^1 \wedge x ^i + y^2 \wedge x^{i+1}\,,
\]
where $(x^0, \dots, x^{j+1}, y^1, y^2)$ denotes the dual basis. Define \[
\mathcal{A} := \{T \in V^\ast \otimes \langle \alpha^i\rangle_{i = 0, \dotsm j} \colon \partial T = 0\}
\]
where $\partial \colon V^\ast \otimes \bigwedge^{2} V^\ast \rightarrow \bigwedge^3 V^\ast$ is given by exterior multiplication (hence, the Spencer coboundary). Then, we have
\[
\mathcal{A} = \left\langle y^1 \otimes \alpha^0 + y^2 \otimes \alpha^1\,, \dots \,, y^1 \otimes \alpha^i + y^2 \otimes \alpha^{i+1}\,, \dots\,,  y^1 \otimes \alpha^{j-1} + y^2 \otimes \alpha^j  \right \rangle\,.
\]
\end{lemma}

\begin{proof} Indeed, let
\[
T = \sum_{i = 0}^{j} \gamma_i\otimes \alpha^i \in V^\ast \otimes \langle \alpha^i \rangle_{i = 0, \dots, j}\,,
\]
where
$
\gamma_i = A_i y^1 + B_i y^2 + C_{im} x^m
$.
Then, we have that
\begin{align*}
    \partial T = & \sum_{i = 0}^{j} \left( A_i y^1 + B_i y^2 + C_{im} x^m\right) \wedge \left( y^1 \wedge x^i + y^2 \wedge x^{i+1}\right)\\
    =& - B_0 y^1 \wedge y^2 \wedge x^0 + \sum_{i = 1}^{j} (A_{i-1} - B_i) y^1 \wedge y^2 \wedge x^i + A_j y^1 \wedge y^2 \wedge x^{j+1}\\
    & +\sum_{i = 0}^j \sum_{m = 0}^{j+1} C_{im}(y^1 \wedge x^i \wedge x^m + y^2 \wedge x^{i+1} \wedge x^m)\,.
\end{align*}
Reorganizing elements we get
\begin{align*}
\partial T =
    & - B_0 y^1 \wedge y^2 \wedge x^0 + \sum_{i = 1}^{j} (A_{i-1} - B_i) y^1 \wedge y^2 \wedge x^i + A_j y^1 \wedge y^2 \wedge x^{j+1}\\
    &  +\sum_{i = 0}^j \sum_{m = 0}^j C_{im} y^1 \wedge x^i \wedge x^m + \sum_{i = 0}^j C_{i\, j+1} y^1 \wedge x^i \wedge x^{j+1}\\
    & + \sum_{i = 0}^j C_{i0} y^2 \wedge x^{i+1} \wedge x^0 +  \sum_{i = 0}^j \sum_{m = 1}^{j} C_{im} y^2 \wedge x^{i+1} \wedge x^m\,.
\end{align*}
If $\partial T = 0$, we obtain the following conditions on the coefficients: First, $B_0 = 0$, $A_j = 0$ and $A_{i-1} = B_i$. The generators of the solutions of this system of equations are the ones prescribed above. Finally, we obtain the following conditions on the $C_{im}$:
\begin{enumerate}[\rm (i)]
    \item For $i = 0, \dots, j$ and $m = 0, \dotsm j$ we have $C_{im} = C_{mi}$.
    \item For $i = 0, \dots, j$ we have $C_{i \, j+1 } = 0$.
    \item For $i = 0, \dots, j$ we have $C_{i0} = 0$.
    \item For $i = 0, \dots, j$ and $ m = 1, \dots, j+1$ we have $C_{im} = C_{m-1 \,i+1}$.
\end{enumerate}
The proof finishes once we show that these conditions imply $C_{im} = 0$, for all $i = 1, \dots, j$ and $m = 0, \dots, j+1$. Indeed, for $m = 1, \dots, j$ and $i = 0, \dots, j$ by combining conditions {\rm (i)} and {\rm (iv)} we have that $C_{im} = C_{i+1\, m-1}$. Now we may show iteratively that all coefficients vanish. Indeed, for $i = 0$, we have that
\[
C_{0m} \stackrel{\text{{\rm (i)}}}{=} C_{m0} \stackrel{\text{{\rm (iii)}}}{=} 0\,,\qquad \text{for $m \leq j$}
\]
and $C_{0 j+1} = 0$ because of {\rm (ii)}. Now suppose that $C_{im} = 0$ for all $i \leq i_0$. Then, by employing the above formula
\[
C_{i_0 +1 \, m}  = C_{i_0\, m} = 0\,, \qquad \text{for $m \leq j$}\,,
\]
the case of $m = j+1$ following from {\rm (ii)}. This finishes the inductive step and concludes the proof.
\end{proof}

Now we explicitly define $\varpi_j$ for $j \geq 1$.

\begin{definition}
\label{def:definition_of_varpi_j}
Let $V$ be a $(2j +5)$-dimensional vector space, together with a basis \[\{x_0, \dots, x_{j+1}, p^0, \dots, p^j, y_1, y_2\}\,,\] and define
\[
\varpi_j := \sum_{i = 0}^j p_i \wedge \alpha^i \,,
\]
where $\{x^0, \dots, x^{j+1}, p_0, \dots, p_j, y^1, y^2\}$ denotes the dual basis and $\alpha^i$ are such as in Lemma \ref{lemma:necessary_lemma}.
\end{definition}

\begin{proof}[Proof of Theorem \ref{thm:Theorem_counterexamples}] Let $V$ and $\varpi \in \bigwedge^3 V^\ast$ be as in Definition \ref{def:definition_of_varpi_j}. We will show that $\mathcal{A}_j$ arises as a retraction of $\mathfrak{g}_{\varpi_j}$. Define
\[
\widetilde V^\ast = \langle y^1, y^2\rangle \subseteq V^\ast \qquad \text{and} \qquad \widetilde W := \langle p^0, \dots, p^j\rangle \subseteq W\,.
\]
We will prove that $\widetilde{\mathcal{A}}:=\mathcal{A} \cap (\widetilde V^\ast \otimes W) \cong \mathcal{A}_j$ and that the natural decompositions
\begin{align*}
    &V^\ast = \widetilde V^\ast \oplus \widehat V^\ast \,, \qquad \text{where} \qquad \widehat{V}^\ast = \langle x^0, \dots, x^{j+1}, p_0, \dots, p_j\rangle \qquad \text{and}\\
    &W = \widetilde W \oplus\widehat W \,, \qquad \text{where} \quad \widehat W= \langle x_0, \dots, x_{j+1}, y_0, y_1 \rangle\,,
\end{align*}
are in the hypotheses of Lemma \ref{lemma:sufficient_retraction_condition} so that the induced mapping
\[
H^{2,j}(\mathcal{A}_j) \xrightarrow{(i_1, i_2)_\ast} H^{2,j}(\mathfrak{g}_{\varpi_j})
\]
is injective. Since, in light of Theorem \ref{theorem:cohomology_in_degree_2} we have $H^{2,j}(\mathcal{A}_j) = \mathbb{R}$, we will obtain the result. Indeed,
\begin{enumerate}[\rm (i)]
    \item \underline{$\widetilde{\mathcal{A}} \cong \mathcal{A}_j$}: By definition, we have that \[\widetilde{\mathcal{A}} = \{T \in \langle y^1, y^2\rangle \otimes \langle p_0, \dots, p_j\rangle \colon \iota_T \varpi_j = 0\}\,.\]
    Let
    $
    T =\sum_{i = 0}^j \gamma_i \otimes p_i \in \widetilde{\mathcal{A}}\,.
    $
    Then we should have
    \[
    0 = \iota_T \varpi_j = \sum_{i = 0}^{j} \gamma_i \wedge \alpha^i\,.
    \]
    Applying Lemma \ref{lemma:necessary_lemma} we obtain
    $
    \widetilde{\mathcal{A}} = \langle y^1 \otimes p_0 + y^2 \otimes p_1\,, \dots\,, y^1 \otimes p_{j-1} + y^2 \otimes p_j\rangle\,,
    $
    which is evidently isomorphic to $\mathcal{A}_j$.
    \item \underline{$\widetilde{\mathcal{A}}$ is a retraction of $\mathcal{A}$}: By Lemma \ref{lemma:sufficient_retraction_condition} it is enough to show that
    \begin{equation}
    \label{eq:condition_on_retraction}
            \mathcal{A}  = \widetilde{\mathcal{A}} \oplus \mathcal{A} \cap \left(\widehat V^\ast \otimes \widetilde W + \widehat{V}^\ast \otimes \widehat W + \widetilde V^\ast \otimes \widehat{W} \right)\,.
    \end{equation}
    To do this, we simply compute $\mathfrak{g}_{\varpi_j}$. Let
    \[
    T = \widetilde T + \gamma_j \otimes p^j + p_i \otimes \left(A^i_j p^j + B^{im} x_m + C^{i}y_0 + D^i y_1 \right)
    \]
    such that $\iota_T \varpi = 0$. Here $\widetilde T$ denotes a tensor entirely on the $x$ and $y$ and $\gamma_j$ denotes a $1$-form on $x$ and $y$. Then, a straightforward computation yields
    \begin{align*}
        0 = \iota_T \varpi =& p_j \wedge \iota_{\widetilde T} \alpha^j + A^{j}_i p_j \wedge \alpha^i + \gamma_j \wedge \alpha^j\\
        &- B^{im}p_i \wedge p_j \wedge (\delta_m^{j} y^1
        + \delta_{m}^{j+1} y^2) + p_i \wedge p_j \wedge ( C^i x_j + D^i x_{j+1})\,,
    \end{align*}
    and we obtain the conditions
    \begin{align*}
        &\iota_{\widetilde T} \alpha_i = -A^j_i \alpha_j \,,  &\gamma_j \wedge \alpha_j = 0\,,\\
        &- B^{im}p_i \wedge p_j \wedge (\delta_m^{j} y^1
        + \delta_{m}^{j+1} y^2) = 0\,,   &p_i \wedge p_j \wedge ( C^i x_j + D^i x_{j+1}) = 0\,.
    \end{align*}
    Since the part on the $\gamma_j$ is not involved in any other equation, we have that Eq.\eqref{eq:condition_on_retraction}, holds finishing the proof.
\end{enumerate}
\end{proof}

\subsection{A multisymplectic manifold with nonvanishing structure tensor of order \texorpdfstring{$j$}{j}}

Here we finish the study by explicitly constructing a multisymplectic form on $M = \mathbb{R}^{5 + 2j}$ with constant linear type $\varpi_j$ for which the structure tensors $c_0, \dots, c_{j-1}$ vanish, but $c_j$ does not. If we endow $\mathbb{R}^{5 + 2j}$ with coordinates $(y^1, y^2, x^0, \dots, x^{j+1}, p_0, \dots, p_j)$, the form will be

\[
\omega_j = \sum_{i = 0}^j \d p_i \wedge \left(\d y^1 \wedge \d x^{i} + \d y^2 \wedge \d x^{i+1}\right) - f(y^2) \d y^1 \wedge \d y^2 \wedge \d x^{j+1}\,,
\]
for some smooth function $f \in C^\infty(\mathbb{R})$.

\begin{remark} The $3$-form $\omega_j$ is of constant linear type $\varpi_j$, since we can write it as
\[
\omega_j = \sum_{i = 0}^{j-1} \d p_i \wedge \left( \d y^1 \wedge \d x^i + \d y^2 \wedge \d x^{i+1}\right) + (\d p_j - f(y^2) \d y^1) \wedge (\d y^1 \wedge \d x^j + \d y^2 \wedge \d x^{j+1})\,.
\]
In fact, this implies that the following frame
\[
    P^i := \pdv{p_i}\,, \qquad
    X_i := \pdv{x^i}\,, \qquad
    Y_1 := \pdv{y^1} + f(y^2) \pdv{p_{j}}\,, \qquad
    Y_2 := \pdv{y^2}\,,
\]
is an adapted frame to the $G_{\varpi_j}$-structure.
\end{remark}

In fact, we show:

\begin{theorem}
\label{thm:explicit_counterexample}
The structure tensors of order $0$ to $j-1$ of the multisymplectic form $\omega_j$ vanish, and the last structure tensor is (represented by the cohomology class of) $0$ if and only if $\pdv{^j f}{(y^2)^j} = 0$.
\end{theorem}

\begin{proof}
    We simply sketch the computations, which are somewhat straightforward. Let us first notice that the linear type
    \[
    \varpi_j = \sum_{i = 0}^j p_i \wedge (y^1 \wedge x^i + y^2 \wedge x^{i+1})
    \]
    is of $2$-isotropic type with $W = \langle x_i, p^i \rangle$. Indeed, a quick computation shows that an element in $\mathfrak{g}_{\varpi_j}$ cannot have summands of the type $p_i \otimes y_i$ or $x^i \otimes y_i$. Hence, we are in the hypotheses of Proposition \ref{prop_sufficiency_condition_solving_flatnes} and if it were flat, there would be functions $f_i(y)$ and $g^i(y)$ such that the coordinate change
    \[
    \widetilde y^i = y^i\,, \qquad \widetilde p_i = p_i - f_i \,, \qquad \widetilde x^i = x^i - g^i
    \]
    makes $\omega$ flat. Imposing that this coordinate change takes $\omega$ to
    \[
    \omega = \sum_{i = 0}^j \d \widetilde p_i \wedge (\d \widetilde y^1 \wedge \d \widetilde x^{i} + \d \widetilde y^2 \wedge \d \widetilde x^{i+1})\,,
    \]
    we obtain the conditions $\dd g^i = 0$ (so that we may assume $g^i = 0$) and the conditions on the $f^i$:
    \[
    \pdv{f^0}{y^2} = 0\,, \pdv{f^0}{y^1} - \pdv{f^1}{y^2} = 0\,, \dots\,, \pdv{f^i}{y^1} - \pdv{f^{i+1}}{y^2} = 0\,, \dots, \pdv{f^{j-1}}{y^1} - \pdv{f^j}{y^2} = 0\,, \pdv{f^j}{y^1} = f(y^2)\,.
    \]
    This equation can be solved up to order $j-1$ with no difficulties, as the choices $f^{i} = \frac{(y^1)^{j+1-i}}{(j+1-i)!} \frac{(y^2)^{i-m}}{(i-m)!}$, for $i \geq m$ and $f^i = 0$, for $i < m$ solve the equations for the monomial $f(y^2) = \frac{(y^2)^{j-m}}{(j-m)!}$, when $m >1$. However, in general, there is no solution to order $j$. Indeed, notice that
    \begin{align*}
        \pdv{^j}{(y^1)^j} \left( - \pdv{f^0}{y^1}\right) + \pdv{^j}{(y^1)^{j-1}y^2} \left( \pdv{f^0}{y^1} - \pdv{f^1}{y^2}\right) + \cdots + \pdv{^j}{(y^1)^{j-i}(y^2)^i} \left( \pdv{f^i}{y^1} - \pdv{f^{i+1}}{y^2}\right)\\
        + \pdv{^j}{(y^1)(y^2)^{j-1}} \left( \pdv{f^j}{y^1} - \pdv{f^{j-1}}{y^2}\right) + \pdv{^j}{(y^2)^j} \left( \pdv{f^j}{y^1}\right) = 0\,,
    \end{align*}
    and if $f^i$ are solutions this would imply
    \[
    \pdv{^{j} f}{(y^2)^j} = 0\,,
    \]
    which proves the result.
\end{proof}

\section{Polarized multisymplectic forms}

\label{section:polarized_form}

In this section we give a complete characterization of $G_\varpi$, for the particular kind of $\varpi$ appearing in classical field theories, which are a generalization of the group appearing in Example \ref{example:Multicotangent_bundle}. We focus on giving a explicit description of this Lie group and computing all of the extensions. We will see that these forms, which are parametrized by three parameters $\varpi_{m,n,r}$ are of $1$-isotropic type so that (in light of Lemma \ref{lemma:semi_finite_condition}), the leaf geometry is actually of finite type.

The forms are defined as follows. Let $\pi: Y \rightarrow X$ be a surjective linear mapping between two vector spaces of dimension, $n+m$ and $n$, respectively. Take
\[V := Y \oplus \bigwedge^n_r Y^\ast\] and define the canonical multisymplectic form of degree $n$ on $V$ as
\[
\varpi_{n,m,r}((y_1, \alpha_1), \dots, (y_{n+1}, \alpha_{n+1})) := \sum_{j = 1}^n (-1)^j \alpha_j(y_1, \dots, \hat{y}_j, \dots, y_{n+1}).
\]

\begin{definition}[The multisymplectic group of type $(m,n,r)$] The \textbf{multisymplectic group of type $(m,n,r)$} is defined as \[
\Mult(m,n,r) := G_{\varpi_{m,n,r}} = \{g \in \GL(V) \colon g^\ast \varpi_{m,n,r} = \varpi_{m,n,r}\}\,.
\]
\end{definition}

\begin{definition}[The multisymplectic Lie algebra of type $(m,n,r)$] The associated Lie algebra to $\Mult(m,n,r)$ is denoted by $\mult(m,n,r)$.
\end{definition}

By Proposition \ref{prop:description_of_Lie_algebra}, we have
\[
\mult(m,n,r) = \{T \in V^\ast \otimes V \colon \iota_{T} \varpi_{m,n,r} = 0\}\,.
\]

The rest of this section is devoted to give an explicit description of $\mult(m , n , r)$, for arbitrary values of $(m , n , r)$. It should be noted beforehand that there will be a fundamental difference between $r \leq m$ and $r > m$, essentially distinguishing between the `fibered' and `multicontangent' type. Furthermore, there are some `degenerate' cases, those corresponding to $r = 0,1$, and $n= 1 , r \geq 2$. Let us first state the result, and then move to discuss each case:

\begin{theorem}[Non degenerate cases]
\label{theorem:non-degenerate-cases}
In the nondegenerate case we have the following:
\begin{enumerate}[\rm (i)]
    \item For $n \geq 2$ and $r > m$, we have that $\mult(m,n ,r)$ is an abelian extension of $\gl(n+m)$.
    \item For $n \geq 2$ and $1 < r \leq m$, $\mult(m , n, r)$ is an abelian extension of $\block(n,m)$.
\end{enumerate}
\end{theorem}

\begin{remark} Here, $\block(n, m)$ denotes the Lie algebra of $\operatorname{Block}(n, m)$, the group of linear isomorphisms of $\mathbb{R}^{n + m}$ that preserve the projection $\pi \colon \mathbb{R}^{n+m} \rightarrow \mathbb{R}^n$. Equivalently, lower block-triangular matrices.
\end{remark}

\begin{remark} In Remark \ref{remark:explicit_description_extension}, we will give explicit description of the abelian extensions, which will be closely related to the kernels of the Spencer differentials
\[
\partial \colon Y^\ast \otimes \bigwedge^{k} Y^\ast \rightarrow \bigwedge^{k+1} Y^\ast\,.
\]
\end{remark}

\begin{theorem}[Degenerate cases]
\label{Thm:Classification_Lie_Algebras}
In the degenerate cases we have the following:
\begin{enumerate}[\rm (i)]
    \item \label{Thm:Classification_Lie_Algebras:r = 0} $\mult(m,n,0) = \gl(m+n)$.
    \item \label{Thm:Classification_Lie_Algebras: r = 1}$\mult(m , n , 1)$ is an abelian extension of $\mathfrak{sl}(n+1) \times \gl(m)$.
    \item \label{Thm:Classification_Lie_Algebras: n = 1}$\mult(m,1,r) = \mathfrak{sp}(m+1)$, for $r \geq 2.$
\end{enumerate}
\end{theorem}

\begin{remark} The case when $n = 1$ is used for the extended formalism of classical mechanics and thus recovers the symplectic group. The non-degenerate cases are the ones with importance for the study of classical field theories. It is usual to take $r = 2$, $n$ to be the dimension of spacetime (and thus is natural to restrict to the case $n \geq 2$), and $m$ to be the dimension of the fibers (the number of field variables). Then, we have two cases. If $m = 1$ (corresponding to a real scalar field) $\mult(1, n , 2)$ is an abelian extension of $\gl(n + 1).$ We will see that this reflects the fact that there are symmetries of a real scalar theory which do not project onto the spacetime. On the other hand, if $m \geq 2$, then $\mult(m ,n , 2)$ is an abelian extension of $\block(n , m)$ and, contrary to the previous case, this will imply that \textit{every} symmetry of the theory projects onto a symmetry of spacetime.
\end{remark}

Let us first prove the immediate cases:

\begin{proof}[Proof of \cref{Thm:Classification_Lie_Algebras}]
Both {\rm (i)} and {\rm (iii)} follow in a straightforward manner. In the first scenario we have $V = Y$ and $\varpi_{m,n,0} = 0$. In the latter,we have $V = Y \oplus Y^\ast$ and the canonical form is the natural symplectic form on $V$. The least obvious case is {\rm (ii)}, but we may compute it directly. Indeed, if ${x_\mu, y_i}$ denote a basis on $Y$ adapted to the epimorphism $\pi \colon Y \rightarrow X$, we have that $V = Y \oplus \bigwedge^{n}_1 Y^\ast \cong Y \oplus \bigwedge^n X^\ast$, so that we have coordinates $\{x_\mu, y_i, p\}$ on $V$ and
\[
\varpi_{m,n,1} = p \wedge x^1 \wedge \cdots \wedge x^n \,,
\]
where $\{x^\mu, y^i, p\}$ denote the dual basis. Then, if
\[
T = \widetilde T + y^i \otimes \left( A_i^\mu x_\mu + B_i p + C_i^j y_j \right) \in V^\ast \otimes V + \gamma^i \otimes y_i\,,
\]
where $\widetilde T$ is a $(1,1)$ tensor on $\{x_\mu, p\}$ only and $\gamma^i$ is a $1$-form on $\{x_\mu, p\}$, we have that $T \in \mult(m,n,1)$ if and only if
\[
0 = \iota_T \varpi_{m,n,1} = \iota_{\widetilde T} (p \wedge x^1 \wedge \cdots \wedge x^n) + B_i y^i \wedge x^1 \wedge \cdots \wedge x^n - (-1)^{\mu} A_i^\mu y^i \wedge p \wedge  x^1 \wedge \cdots \wedge \hat{x}^\mu \wedge \wedge x^n = 0\,.
\]
In particular, $\widetilde T \in \mathfrak{sl}(n+1)$, and $A_i^\mu = B_i = 0$. It follows that $\mult(m,n,1)$ is an abelian extension of $\mathfrak{sl}(n+1) \times \gl(m)$.
\end{proof}

In order to prove the nondegenerate cases, it will (mostly) suffice to observe the following:

\begin{lemma}
\label{lemma:W_invarant}
Define
$
W := \{0\} \oplus \bigwedge^k_r Y^\ast \subsetneq V\,.
$
Then, $W$ is an invariant under the action of $\Mult(m,n,r)$, in the nondegenerate case. Furthermore, if
$
\widetilde W := \ker \pi \oplus \bigwedge^k_r Y^\ast\,,
$
and $1 < r \leq m$, then $\widetilde W$ is also $\mult(m,n, r)$-invariant.
\end{lemma}
\begin{proof} We simply take an arbitrary element in $T \in V^\ast \otimes T$ and impose $\iota_T \varpi_{k,m,r} = 0$. Take an adapted basis $(x_\mu, y_i)$ and let $p^{\mu_1, \dots, \mu_m i_{m+1} \dots i_n}$ denote $p^{\mu_1, \dots, \mu_m i_{m+1} \dots i_n} = x^{\mu_1} \wedge \cdots \wedge x^{\mu_m} \wedge y^{i_1} \wedge \cdots \wedge y^{i_n}$, for $m = n+1-r, \dots n$, this defines a basis of $V$ in which the multisymplectic form is expressed as
\[
\varpi_{m,n,r} = p^{1, \dots, n} x^{1} \wedge \cdots \wedge x^n + \cdots + p_{\mu_1 \dots \mu_{n+1-r} i_{n+2-r} \dots i_n} \wedge x^{\mu_1} \wedge \cdots \wedge x^{\mu_{n+1-r}} \wedge y^{i_{n+2-r}} \wedge \cdots \wedge y^{i_n}
\]
From this expression, it is not hard (although a bit tedious) to find that if $T \in V^\ast \otimes V$ is such that $\iota_T \varpi_{m,n,r} = 0$, we have $T(W) \subseteq W$. The general idea is to look for the terms involving two $p$'s, and finding that the vanishing of these terms implies the vanishing of the coefficients on $T$ that would make $T(W) \not \subseteq W$.
\end{proof}

From this, we can prove Theorem \ref{theorem:non-degenerate-cases}:

\begin{proof}[Proof of Theorem \ref{theorem:non-degenerate-cases}] By taking $L = Y \oplus \{0\}$, and in light of Lemma \ref{lemma:W_invarant} we have that the decomposition $V = L \oplus W$ is in the hypotheses of Proposition \ref{prop:explicit_description_of_alegbra_1isotropic}. In particular, we obtain a semi-direct product decomposition
\[
\mathfrak{g}_{\varpi_{m,n,r}} \cong \mathfrak{g}_S \ltimes \mathcal{A}_{\varpi_{m,n,r}}\,.
\]
By definition, $\mathfrak{g}_{\varpi_{m,n,r}}$ is an abelian extension of $\mathfrak{g}_S$ which is:
\begin{enumerate}[\rm (i)]
    \item If $r > m$, we have that $S = \bigwedge^{k-1} Y^\ast$ so that $\mathfrak{g}_S = \gl(Y) \cong \gl(n+m)$.
    \item If $1 < r \leq m$, we have that $S = \bigwedge^{k-1}_r Y^\ast$, and it is a straightforward to check that $\mathfrak{g}_S$ is precisely the space of endomorphisms $T \colon Y \rightarrow Y$ commuting with $\pi$. In particular, $\mathfrak{g}_S \cong \block(m, n)$.
\end{enumerate}
\end{proof}

\begin{remark}
\label{remark:explicit_description_extension}
In fact, from the previous proof, we can see that we have en explicit description of the Lie algebra in the non-degenerate cases:
\[
\mult(m,n, r) \cong \begin{cases}
\block(m,n) \ltimes \ker \partial_r \,, & 1 \leq r \leq m\\
\gl(m+n) \ltimes \ker \partial \,, & r > m
\end{cases}\,,
\]
where $\partial$ denotes the usual Spencer differential and
\[
\partial_r \colon Y^\ast \otimes \bigwedge^{k}_r Y^\ast \rightarrow \bigwedge^{k+1}_{r+1} Y^\ast
\]
denotes the restriction to the subspace $\bigwedge^k_r Y^\ast$ of the Spencer differential.
\end{remark}

This, in turn, may be used to compute all prolongations of this Lie algebra which, which we can employ to give a characterization of all Hamiltonian forms of a flat polarized multisymplectic structure. In fact, by Corollary \ref{cor:extensions_in_1_isotropic_case} we have that:

\begin{proposition} For the nondegenerate cases we have
\[
\mult(m,n,r)^{(j)} \cong \begin{cases}
\block(m,n)^{(j)} \oplus \ker \partial_r \,, & 1 \leq r \leq m\\
\gl(m+n)^{(j)} \oplus \ker \partial \,, &r > m
\end{cases}\,,
\]
where $\partial$, identified as a map
\[
\partial \colon S^{j+1} (Y^\ast) \otimes \bigwedge^n Y^\ast \rightarrow S^{j}(Y^\ast) \otimes \bigwedge^{n+1} Y^\ast\,,
\]
represents the Spencer coboundary and $\partial_r$ denotes its restriction to $S^{j}(Y^\ast) \otimes \bigwedge^n_r Y^\ast$.
\end{proposition}

\begin{remark} In particular, we have that for the flat multisymplectic structure $\varpi_{m,n,r}$ on $V = Y^\ast \oplus \bigwedge^n_2 Y^\ast$ a Hamiltonian $(n-1)$-form $\alpha \in \Omega^{n-1}_H(V)$ can be written, up to a closed term, as
\[
\alpha = \left(p A^\mu(x) + p^\mu_i B^i(x, y) \right) \dd^{n-1}x_\mu - p^\mu_i A^{\nu} \dd y^i \wedge \dd^{n-2} x_{\mu \nu} + \beta\,,
\]
where $\beta \in \Omega^{n-1}_{1}(Y)$ is arbitrary.
\end{remark}


\section{Conclusions and further work}
\label{section:conclusions}

In this work, we have approached multisymplectic structures of constant linear type $\varpi \in \bigwedge^{k+1} V^\ast$ from the point of view of their associated $G_\varpi$-structures. As a consequence, we have obtained a framework that improves our understanding of multisymplectic geometry and of the differential conditions ensuring flatness, namely the existence of a chart in which the multisymplectic form has constant coefficients. Some of the main results are the following:

\begin{enumerate}[\rm (i)]
    \item We show how to construct the structure tensors of multisymplectic manifolds with constant linear type. The vanishing of these tensors is a necessary condition for flatness (which is also sufficient in the analytic setting) and we give several instances in which the converse holds in the smooth setting. We conjecture that this is the case for any smooth multisymplectic structure.
    \item We achieved a rough classification of the linear multisymplectic representations (see Theorem \ref{thm:casuistic_of_representation}), which allows for a systematic study of Darboux theorem for multisymplectic forms $\omega \in \Omega^{k+1}(M)$. We showed that by taking into account this classification (into irreducible, product, $j$-isotropic, and semifinite types), we can extract a lot of information for the study of flatness conditions (see Theorem \ref{thm:sufficient_condition_flatness_general}, Proposition \ref{prop_sufficiency_condition_solving_flatnes} and Theorem \ref{thm:Darboux_1_isotropic}). This was illustrated through several examples appearing in the literature.
    \item We have constructed new models of multisymplectic geometry with properties distinct from those found in the literature (see Theorems \ref{thm:Theorem_counterexamples} and \ref{thm:explicit_counterexample}). Indeed, we have shown that there are geometries for which the flatness conditions are not algebraic conditions on the first-order derivatives of the multisymplectic form (such as requiring it to be closed), but instead involve derivatives of any prescribed order $j \in \mathbb{N}$. These can be realized by the explicit $3$-forms constructed in Theorem \ref{thm:explicit_counterexample}.
    \item Finally, we presented the computations of the Lie algebras and the extensions in the particular case relevant for field theories.
\end{enumerate}

From a more general viewpoint, we would like to highlight the following. From the lens of the theory of $G$-structures, `Multisymplectic geometry' is not a single geometry but a plethora of them. Indeed, very distinct isotropy groups and representations may arise. These representations, as exemplified by the linear types $\varpi_j$, have very different characteristics (order of the geometry, finite vs. infinite type, and irreducible vs. reducible representations). This is in clear opposition to more classic geometries, such as Riemannian, complex, symplectic, Kähler, etc. Thus, we believe that the difficulty of dealing with multisymplectic geometry in full generality is high. However, the theory of $G$-structures is well suited for dealing with this complexity. In particular, the Spencer complex provides systematic tools to understand the local structure of these manifolds. Nevertheless this paper leaves several questions open:
\begin{enumerate}[\rm (i)]
    \item One of the main questions that this work leaves unsolved is the problem of the integrability of $G$-structures and, more particularly, the case where $G = G_\varpi$, for some $\varpi \in \bigwedge^{k+1} V^\ast$. 
    
    \item A better understanding of the ``higher form symmetries'' of field theories and their pointwise and local properties \cite{SciPostPhysLectNotes.74}. These symmetries are described by multivectors or, dually, Hamiltonian forms of degree $< n-1$, where $n$ is the degree of the multisymplectic form, in the line of \cite{di2025j.math.phys.}. Indeed, for the pointwise study, one should understand the relationship between the representation of $G$ on $V$ and its induced representations on exterior powers. The analysis of these forms will also shed light on the local structure of the Hamilton--de-Donder--Weyl equations in field theory, which can be encoded in a higher Poisson algebra.
    
    \item One of the main topics in multisymplectic field theories is the problem of reduction. Although much progress has been made \cite{bla2021lett.math.phys., CALLIES2016954, ms2012advancesinmathematics}, there are still some issues that need to be clarified. A study `per linear type', as presented in this text, could help clarify these issues. The theory of $G$-structures and Spencer cohomology can help to understand the existence of homomorphisms between multisymplectic manifolds.

    \item In a more geometric vein, in symplectic geometry there are normal forms for Lagrangian submanifolds \cite{wei1971adv.math.}. Although similar results have been proven for specific linear types \cite{aitor1}, the study of normal forms in arbitrary multisymplectic manifolds is usually ignored. This problem could be dealt with by studying which subspaces $W \subseteq V$ of a particular multisymplectic vector space lie in the same orbit under the group $G_\varpi$. It is clear (unlike in symplectic geometry) that the classification of subspaces into $j$-isotropic, $j$-coisotropic and $j$-Lagrangian is insufficient. Moreover, this finer classification can be useful for advancing the problem of integrability of multisymplectic structures.

     \item Among the zoo of multisymplectic geometries that can be defined, probably not all of these will be relevant in field theories. In order to see which ones appear in this setting, we can follow two complementary routes. One is to analyze relevant multisymplectic $G$-structures appearing in known field theories, such as gauge theories, regular field theories of arbitrary order, and outstanding singular theories. The other is the study of the formal properties that render them unphysical, such as determining which ones have the expected amount of solutions to the Hamilton--de-Donder--Weyl equations, say, to make the Cauchy problem well-posed. 
\end{enumerate}

\begin{appendices}

\section{Smooth solutions of inhomogeneous systems of partial differential equations with constant coefficients}

\label{Appendix:Malgrange_Theorem}

Inhomogeneous compatible linear systems of PDEs are known to have solutions. The Malgrange--Ehrempreis theorem~\cite{Malgrange1956,Ehrenpreis1954} state that it has a Green function, thus one is able to construct solutions on several function spaces (see~\cite{Palamodov1970}).  We have been unable to find a clear statement in the literature in the smooth setting. Nevertheless, it is a consequence of the $C^k$ version found in
\cite[Theorems~7.6.13--7.6.14]{Hormander1990}. These results say as follows:
\begin{theorem}[Hörmander, Theorem 7.6.13]
\label{thm:Hormander}
Let \(\Omega\subset\mathbb R^n\) be open and convex. Then there is an
integer \(N\) such that the system
\[
    P(D)u=f
\]
has a solution
\[
    u\in C^\nu(\Omega)^K
\]
for every
\[
    f\in C^{\nu+N}(\Omega)^J
\]
satisfying the compatibility equations
\[
    Q(D)f=0.
\]
\end{theorem}

Here, by exponential solution it is meant a solution of the form
\[
    u(x)=e^{i\langle x,\zeta\rangle}u_0(x),
    \qquad \zeta\in\mathbb C^n,
\]
where the components of \(u_0\) are polynomials. Indeed, we only need the fact that there is a family of dense smooth solutions of the homogeneus problem. With this, we can construct a $C^\infty$ version of the theorem.
\begin{theorem}[Hörmander, Theorem 7.6.14]
Let \(\Omega\subset\mathbb R^n\) be open and convex. Then the closure in
\(C^\nu(\Omega)^K\) of the linear combinations of exponential solutions of
the homogeneous system
\[
    P(D)u=0
\]
is equal to the space of all \(C^\nu(\Omega)^K\)-solutions of this system.
\end{theorem}
\begin{theorem}
Let $\Omega\subset \mathbb R^n$ be an open convex set. Consider the
constant-coefficient system
\[
    P(D)u=f,
    \qquad
    P(D)=\bigl(P_{jk}(D)\bigr)_{1\leq j\leq J,\ 1\leq k\leq K}.
\]
Let \(Q(D)\) denote the corresponding compatibility operator, so that
\(Q(D)P(D)=0\). If
\[
    f\in C^\infty(\Omega)^J,
    \qquad
    Q(D)f=0,
\]
then there exists
\[
    u\in C^\infty(\Omega)^K
\]
such that
\[
    P(D)u=f.
\]
\end{theorem}

\begin{proof}
We use the usual Fréchet topology on \(C^\infty(\Omega)^K\) of uniform convergence of compact sets of all the derivatives: a sequence
\(w_\ell\) converges to \(w\) if and only if, for every compact set
\(L\subset \Omega\) and every \(m\in\mathbb N\),
\begin{equation}
    \max_{1\leq k\leq K}
    \max_{|\alpha|\leq m}
    \sup_{x\in L}
    \left|
        \partial^\alpha (w_{\ell,k}-w_k)(x)
    \right|
    \longrightarrow 0.
\end{equation}

Equivalently, it is enough to use any increasing sequence of compact sets
whose interiors cover \(\Omega\), since each compact set hast to be contained in one of the sequence.

We now construct such a sequence explicitly. Fix \(a\in\Omega\). Choose $0<\rho<\operatorname{dist}(a,\mathbb R^n\setminus\Omega)$,
and let
\begin{equation}
    E_i
    :=
    \left\{
        x\in\mathbb R^n:
        |x-a|\leq i,\quad
        \operatorname{dist}(x,\mathbb R^n\setminus\Omega)
        \geq \frac{\rho}{i+1}
    \right\}.
\end{equation}
Then each \(E_i\) is compact and convex. Indeed, the condition $\operatorname{dist}(x,\mathbb R^n\setminus\Omega)
    \geq \delta$
is equivalent to $\overline{B(x,\delta)}\subset \Omega$,
and the set of such points is convex because \(\Omega\) is convex. Intersecting
with the closed ball \(\{|x-a|\leq i\}\) gives a compact convex set.

Moreover, $E_i\subset E_{i+1}$,
and the interiors of the \(E_i\)'s cover \(\Omega\). Indeed, if \(x\in\Omega\),
then we let $d=\operatorname{dist}(x,\mathbb R^n\setminus\Omega)>0$,
so for \(i\) large enough we have
\[
    |x-a|<i,
    \qquad
    \frac{\rho}{i+1}<d,
\]
hence \(x\) belongs to the interior of \(E_i\).

For \(w\in C^m(\Omega)^K\), define the seminorm
\[
    \|w\|_{m,E_i}
    :=
    \max_{1\leq k\leq K}
    \max_{|\alpha|\leq m}
    \sup_{x\in E_i}
    |\partial^\alpha w_k(x)|.
\]
These seminorms define the \(C^\infty\)-topology above.

Let $p:=\max_{j,k}\operatorname{ord}P_{jk}$.
Fix \(\epsilon>0\), and put $\epsilon_i:=2^{-i}\epsilon$.

We shall construct a sequence of functions
$u_i\in C^{i+p}(\Omega)^K$
such that $P(D)u_i=f$ and $u_{i+1}-u_i\|_{i,E_i}<\epsilon_i$. Finally, we will see that these converges to a smooth solution on $\Omega$.

For \(i=1\), \cite[Theorem 7.6.13]{Hormander1990} applied with $\nu=1+p$ gives $u_1\in C^{1+p}(\Omega)^K$
such that $P(D)u_1=f$.

This is applicable because \(f\in C^\infty(\Omega)^J\) and \(Q(D)f=0\).

Assume that \(u_i\) has already been constructed. Applying Theorem 7.6.13
with  $\nu=i+1+p$,
we obtain $v_{i+1}\in C^{i+1+p}(\Omega)^K$
such that $P(D)v_{i+1}=f$.

Define
\begin{equation}
    r_i:=u_i-v_{i+1}.
\end{equation}
Then, $r_i\in C^i(\Omega)^K$ and
\[
    P(D)r_i
    =
    P(D)u_i-P(D)v_{i+1}
    =
    f-f
    =
    0.
\]
By Theorem 7.6.14, the finite linear combinations of exponential solutions
of the homogeneous system
\[
    P(D)s=0
\]
are dense in the space of \(C^i\)-solutions of \(P(D)s=0\), with respect to
the \(C^i(\Omega)^K\)-topology. Hence there exists such an exponential (thus smooth)
solution \(s_i\) satisfying
\[
    \|s_i-r_i\|_{i,E_i}<\epsilon_i .
\]
Now set
\[
    u_{i+1}:=v_{i+1}+s_i .
\]
Since \(s_i\) is smooth and satisfies \(P(D)s_i=0\), we have $u_{i+1}\in C^{i+1+p}(\Omega)^K$
and
\[
    P(D)u_{i+1}
    =
    P(D)v_{i+1}+P(D)s_i
    =
    f.
\]
Furthermore,
\[
\begin{aligned}
    \|u_{i+1}-u_i\|_{i,E_i}
    &=
    \|v_{i+1}+s_i-u_i\|_{i,E_i}      \\
    &=
    \|s_i-(u_i-v_{i+1})\|_{i,E_i}    \\
    &=
    \|s_i-r_i\|_{i,E_i}              \\
    &<
    \epsilon_i .
\end{aligned}
\]
This completes the induction.

We now pass to the limit. Let \(L\subset\Omega\) be compact and let
\(m\in\mathbb N\). Since the interiors of the \(E_i\)'s cover \(\Omega\),
and since the \(E_i\)'s are increasing, there exists an index \(\ell\) such
that  $L\subset E_\ell$.
If $j>k\geq \max\{\ell,m\}$,
then all terms \(u_i\), \(i\geq k\), have at least \(m\) derivatives on
\(\Omega\), and
\[
\begin{aligned}
    \|u_j-u_k\|_{m,L}
    &\leq
    \sum_{i=k}^{j-1}
    \|u_{i+1}-u_i\|_{m,L}             \\
    &\leq
    \sum_{i=k}^{j-1}
    \|u_{i+1}-u_i\|_{i,E_i}           \\
    &<
    \sum_{i=k}^{j-1}
    \epsilon_i  =     \sum_{i=1}^{\infty}2^{-i}\epsilon
    =
    \epsilon.
\end{aligned}
\]
Thus, the tail of \((u_i)\) is Cauchy in \(C^m(L)^K\).

Taking \(m=0\), we first obtain a continuous function $u:\Omega\longrightarrow \mathbb C^K$

such that
\[
    u_i\longrightarrow u
\]
uniformly on every compact subset of \(\Omega\). Now fix \(m\geq 1\). The
same argument shows that, on every compact \(L\subset\Omega\), the tail of
\((u_i)\) converges in \(C^m(L)^K\) to some \(C^m\)-limit. Since convergence
in \(C^m(L)^K\) implies convergence in \(C^0(L)^K\), this \(C^m\)-limit must
be the same function \(u\). Hence $u\in C^m(\Omega)^K$ for every $m$. Therefore
$u\in C^\infty(\Omega)^K$, and $u_i\longrightarrow u$
in \(C^m(L)^K\) for every compact \(L\subset\Omega\) and every \(m\).

Finally, let \(m\geq p\), where $p=\max_{j,k}\operatorname{ord}P_{jk}$.
For every compact \(L\subset\Omega\), we have
    $u_i\longrightarrow u
    \quad\text{in }C^m(L)^K.$

Since \(P(D)\) has order at most \(p\), it follows that
\[
    P(D)u_i\longrightarrow P(D)u
    \quad\text{in }C^{m-p}(L)^J.
\]
But
\[
    P(D)u_i=f
\]
for every \(i\). Therefore
\[
    P(D)u=f
\]
on \(L\). Since \(L\subset\Omega\) was arbitrary, we conclude that
\[
    P(D)u=f
    \quad\text{in }\Omega .
\]\end{proof}

\end{appendices}

\section*{Acknowledgements}
The authors acknowledge financial support from the Spanish Ministry of Science, Innovation and Universities under grants PID2022-137909NB-C21 and the Severo Ochoa Program for Centers of Excellence in R\&D (CEX2023-001347-S). Rubén Izquierdo-López whises to thank the Spanish Ministry of Science, Innovation and Universities for the contract FPU24/02636.

\phantomsection
\addcontentsline{toc}{section}{References}
\printbibliography
\end{document}